\documentclass[10pt,reqno]{amsart}

\usepackage{a4wide}
\usepackage[active]{srcltx}           
\usepackage[all]{xy}
\usepackage{dsfont}                   
 \usepackage{comment}
\usepackage{mathrsfs}
 \usepackage{amssymb}

\newtheorem{proposition}{Proposition}[section]
\newtheorem{theorem}[proposition]{Theorem}
\newtheorem{lemma}[proposition]{Lemma}

\theoremstyle{definition}
\newtheorem{definition}[proposition]{Definition}
\newtheorem{deft}[proposition]{Definition}

\theoremstyle{remark}
\newtheorem{remark}[proposition]{Remark}

\newcommand{\Pol}{\textup{Pol}}
\newcommand{\ot}{\otimes}
\newcommand{\Com}{\Delta}
\newcommand{\cst}{\ifmmode\mathrm{C}^*\else{$\mathrm{C}^*$}\fi}
\newcommand{\CC}{\mathbb{C}}
\newcommand{\NN}{\mathbb{N}}
\newcommand{\ZZ}{\mathbb{Z}}
\newcommand{\bz}{\mathbb{Z}}
\newcommand{\bc}{\mathbb{C}}
\newcommand{\QG}{\mathbb{G}}
\newcommand{\QH}{\mathbb{H}}
\newcommand{\tens}{\otimes}
\newcommand{\id}{\mathrm{id}}
\newcommand{\comp}{\!\circ\!}
\newcommand{\ph}{\varphi}
\newcommand{\I}{\mathds{1}}
\newcommand{\blg}{\mathsf{B}}
\newcommand{\alg}{\mathsf{A}}
\newcommand{\sS}{\mathsf{S}}
\newcommand{\sA}{\mathsf{A}}

\newcommand{\sB}{\mathsf{B}}
\newcommand{\sH}{\mathsf{H}}
\newcommand{\cA}{\mathscr{A}}

\newcommand{\hh}[1]{\widehat{#1}}
\renewcommand{\Bar}[1]{\overline{#1}}
\newcommand{\sn}{\overset{\scriptscriptstyle{n}}{s}}
\newcommand{\cn}{\overset{\scriptscriptstyle{n}}{c}}
\newcommand{\taun}{\overset{\scriptscriptstyle{n}}{\tau}}
\newcommand{\sigman}{\overset{\scriptscriptstyle{n}}{\sigma}}
\newcommand{\snn}{\overset{\scriptscriptstyle{n+1}}{s}\!\!}
\newcommand{\taunn}{\overset{\scriptscriptstyle{n+1}}{\tau}\!\!\!}
\newcommand{\sigmann}{\overset{\scriptscriptstyle{n+1}}{\sigma}\!\!}
\newcommand{\sigmarnn}{\overset{\scriptscriptstyle{2n+1}}{\sigma}\!\!}
\newcommand{\skk}{\overset{\scriptscriptstyle{2k}}{s}\!}

\newcommand{\skkj}{\overset{\scriptscriptstyle{2k+1}}{s}\!\!}
\newcommand{\pr}{\mathrm{pr}}
\newcommand{\ev}{{\scriptscriptstyle\text{\rm{even}}}}
\newcommand{\odd}{{\scriptscriptstyle\text{\rm{odd}}}}
\newcommand{\Ind} {\mathcal{I}}
\newcommand{\catV}{\mathcal{C}_{\sA,\mathcal{V}}}
\newcommand{\Alg}{\mathscr{A}}
\DeclareMathOperator{\QI}{\mathsf{QISO}}
\DeclareMathOperator{\C}{C}
\DeclareMathOperator{\Lin}{Lin}
\newcommand{\Zpn}{\ZZ/_{p^n \ZZ}}

\newcommand{\Zpnplus}{\ZZ/_{p^{n+1} \ZZ}}
\newcommand{\Pad}{\ZZ_p}
\newcommand{\Zl}{\ZZ/_{l \ZZ}}

\numberwithin{equation}{section}

\begin{document}

\author{Adam Skalski}
\address{Institute of Mathematics of the Polish Academy of Sciences,
ul.~\'Sniadeckich 8, 00--956 Warszawa, Poland}
\email{a.skalski@impan.pl}

\author{Piotr M.~So{\l}tan} \address{Department of Mathematical Methods in Physics, Faculty of Physics, University of Warsaw, Poland}
\email{piotr.soltan@fuw.edu.pl}

\thanks{Partially supported by National Science Centre (NCN) grant no.~2011/01/B/ST1/05011}

\title[Limits of quantum symmetry groups]{Projective limits of quantum symmetry groups and the doubling construction for Hopf algebras}


\begin{abstract}
The quantum symmetry group of the inductive limit of \cst-algebras equipped with orthogonal filtrations is shown to be the projective limit of the quantum symmetry groups of the \cst-algebras appearing in the sequence. Some explicit examples of such projective limits are studied, including the case of quantum symmetry groups of the duals of finite symmetric groups, which do not fit directly into the framework of the main theorem and require further specific study. The investigations reveal a deep connection between quantum symmetry groups of discrete group duals and the doubling construction for Hopf algebras.
\end{abstract}

\maketitle

Within the last fifteen years the fundamental notion of a symmetry group of a given structure (a set, a graph, a finite or compact metric space, a compact Riemannian manifold) has been successfully generalised to the context of \emph{compact quantum groups} of S.L.Woronowicz. The history of these developments started in the paper \cite{WangSym}, where S.\,Wang defined and began to study quantum symmetry groups of finite-dimensional \cst-algebras. Soon after that T.\,Banica, J.\,Bichon and others expanded this study to quantum symmetry groups of various finite structures, such as (coloured) graphs or metric spaces (see references below). In 2009 D.\,Goswami proposed a definition of a quantum isometry group of an admissible spectral triple, the latter playing the role of a noncommutative compact Riemannian manifold. Particular examples of Goswami's quantum isometry groups are those associated to duals of finitely generated groups, first studied in \cite{bhs} and later investigated in \cite{bsk1}, \cite{Symmetric} and \cite{bsk2}. These examples, together with Goswami's definition, motivated the introduction of the concept of a quantum symmetry group of a $\cst$-algebra equipped with an orthogonal filtration (\cite{orth}), which plays a fundamental role in this article. Note at the same time that this concept has been very recently generalized in the preprint \cite{Manon} to include quantum symmetry groups of a $\cst$-\emph{module} equipped with an orthogonal filtration.

The quantum symmetry groups arising in the contexts described above have been studied from a variety of algebraic and analytic points of view. In the current article we investigate two specific topics related to such quantum groups. The first is related to the concept of a doubling of a given compact quantum group with respect to an order two automorphism, formally introduced in \cite{Symmetric}. Such doublings seem to appear often as quantum symmetry groups of duals of finitely generated groups (see \cite{bhs}, \cite{Symmetric}, \cite{Chin}); here we explain some reasons behind that and provide further examples. The second, more important, topic is that of projective limits of quantum symmetry groups. It was investigated in the context of quantum isometry groups, especially those related to Bratteli diagrams, in \cite{bgs}; here we study it systematically in the framework of quantum symmetry groups of $\cst$-algebras equipped with orthogonal filtrations. The classical idea is rather simple: if we want to investigate a symmetry group $G$ of a structure arising as a projective limit of simpler (say finite) structures, it is natural to expect that $G$ will itself be a projective limit of the corresponding symmetry groups. The same phenomenon persists in the quantum world; we show here, in Theorem \ref{mainlimit}, that under suitable conditions on a sequence of \cst-algebras with orthogonal filtrations, the quantum symmetry group of the limit filtration arises as the projective limit of quantum symmetry groups of the \cst-algebras appearing in the sequence. We provide an example of such a projective limit related to the quantum symmetry group of the group of $p$-adic integers (which in fact turns out to be a classical symmetry group). As the assumptions in Theorem \ref{mainlimit} are rather restrictive, it is difficult to expect that they will be always satisfied; indeed, we explain why they fail for the sequence of filtrations related to the duals of symmetric groups, $(\widehat{\sS}_n)_{n=1}^\infty$, studied in \cite{Symmetric}. This however does not mean that the related sequence of quantum symmetric groups cannot lead to a projective limit. We prove that the respective sequences of quantum symmetry groups associated to even  -- $(\widehat{\sS}_{2n})_{n=1}^\infty$ -- and odd
-- $(\widehat{\sS}_{2n+1})_{n=1}^\infty$ -- permutation groups yield (different!) compact quantum groups as projective limits and identify the limits as quantum symmetry groups of certain filtrations induced by specific partitions of $\sS_{\infty}$. Interestingly, both the limit compact quantum groups arising in the context of the group of $p$-adic integers and in the context of the infinite symmetric group turn out to be doublings of the respective duals of discrete groups.

The detailed plan of the paper is as follows: in Section 1 we introduce the notation and recall basic facts related to compact quantum groups, their projective limits, and the construction of quantum symmetry groups. Section 2 describes the doubling procedure for Hopf algebras, focusing on the algebras related to compact quantum groups, and explains its common appearance in the study of quantum symmetry groups of duals of finitely generated groups (Proposition \ref{doubsym}). Section 3 contains the main general result of the paper related to the projective limits of quantum symmetry groups, Theorem \ref{mainlimit}. In Section 4 Theorem \ref{mainlimit} is `shown in action' via the example of quantum symmetry groups of the group of $p$-adic integers. In Section 5 we investigate limits of quantum symmetry groups related to (the duals of) permutation groups; we show that although the relevant orthogonal filtrations do not fit into the framework of Theorem \ref{mainlimit}, they still lead, when restricted to sequences built respectively of even and odd integers, to compact quantum groups arising as projective limits (Subsection \ref{indlim}). We then show that the two limit compact quantum groups are non-isomorphic and identify them as quantum symmetry groups associated to specific partitions of $\sS_{\infty}$ (Theorems \ref{nonisomThm}, \ref{oddSymmetry} and \ref{evSymmetry}).

\section{Preliminaries}
All algebras and \cst-algebras in this paper will be assumed \textbf{unital}, and all $*$-homomorphisms between algebras will be assumed to be \textbf{unit preserving}. The symbol $\ot$ will denote the minimal/spatial tensor product of \cst-algebras and the tensor product of linear maps (or its continuous extension to completed tensor products), whereas $\odot$ will be reserved for algebraic tensor products.

\subsection{Compact quantum groups and quantum symmetry groups}\label{prelCQG}

The following definition is due to S.L.Woronowicz (\cite{wor}).

\begin{deft}
A unital \cst-algebra $\blg$ equipped with a unital $*$-homomorphism $\Com:\blg \to \blg \ot \blg$ is said to be the algebra of functions on a compact quantum group if
$\Com$ is coassociative:
\[ (\id \ot \Com) \circ \Com= (\Com \ot \id) \Com\]
and the quantum cancellation rules hold:
\[ \overline{\Lin}\,\left( \Com(\blg)(1 \ot \blg) \right)= \overline{\Lin}\,\left( \Com(\blg)( \blg \ot 1)\right) = \blg \ot \blg.\]
We then often write $\blg=\C(\QG)$ and call the (formally undefined) underlying object $\QG$ a compact quantum group.
\end{deft}

Each compact quantum group $\QG$ admits a \emph{Haar state}, i.e.\ a state $h \in \C(\QG)^*$ which is bi-invariant:
\[ (h\ot \id)\circ \Com = (\id \ot h) \circ \Com = h(\cdot)1.\]
This leads to a version of the Peter-Weyl theory for compact quantum groups: in particular for each compact quantum group $\QG$ the algebra $\C(\QG)$ admits a dense unital $*$-subalgebra $\Pol(\QG)$, which is a Hopf $*$-algebra with the coproduct inherited from $\C(\QG)$, so that in particular $\Com(\Pol(\QG)) \subset \Pol(\QG) \odot \Pol(\QG)$. The algebra $\Pol(\QG)$ in a sense determines $\QG$ uniquely, although there is a certain ambiguity as to the choice of $\C(\QG)$ -- it can be for example given by the enveloping \cst-algebra of $\Pol(\QG)$ or as the completion of $\Pol(\QG)$ with respect to the GNS representation relative to the Haar state.


\begin{deft}
Let $\QG_1, \QG_2$ be compact quantum groups. By a morphism from $\QG_1$ to $\QG_2$ we understand a $*$-homomorphism $\gamma:\C(\QG_2) \to \C(\QG_1)$ intertwining the respective coproducts:
\[ (\gamma \ot \gamma)\circ \Com_2 = \Com_1 \circ \gamma\]
(note the inversion of arrows).
\end{deft}

Each morphism as above automatically maps $\Pol(\QG_2)$ into $\Pol(\QG_1)$, and the corresponding restriction is a Hopf $*$-algebra morphism. In particular we say that two compact quantum groups $\QG_1$, $\QG_2$ are \emph{isomorphic} if Hopf $*$-algebras $\Pol(\QG_1)$ and $\Pol(\QG_2)$ are isomorphic. Further we say that $\QG_1$ is a \emph{(closed) quantum subgroup} of $\QG_2$ if there exists a morphism from $\QG_1$ to $\QG_2$ such that the corresponding \cst-algebra morphism $\gamma$ is surjective (for a full discussion of the concept of a closed quantum subgroup, including the context of locally compact quantum groups, we refer to \cite{DKSS}).

\begin{deft} \label{actdef}
Let $\QG$ be a compact quantum group and let $\sA$ be a \cst-algebra. An action of $\QG$ on $\alg$ is a $*$-homomorphism $\alpha:\sA \to \sA \ot \C(\QG)$ which satisfies the action equation:
\[ (\alpha \ot \id) \circ \alpha = (\id \ot \Com) \circ \alpha\]
and the Podle\'s condition (sometimes called also the nondegeneracy condition):
\[ \overline{\Lin} \, \alpha(\sA) (1 \ot \C(\QG)) = \sA \ot \C(\QG).\]
\end{deft}

If $\sA=\C(X)$ for a compact space $X$, we sometimes say that $\alpha$ as above defines an action of $\QG$ on $X$. Each action of $\QG$ on a unital \cst-algebra $\sA$ has a purely algebraic incarnation -- there exists $\mathcal{A}$, a dense unital $*$-subalgebra of $\sA$, such that $\alpha|_{\mathcal{A}}: \mathcal{A} \to \mathcal{A} \odot \Pol(\QG)$. The action is said to be \emph{faithful} if the algebra $R_{\alpha}(\QG)$ spanned by the elements of the form $(f \ot \id)\alpha(a)$ ($f$ - linear functional on $\mathcal{A}$, $a \in \mathcal{A}$) is equal to $\Pol(\QG)$.

The collection of actions of compact quantum groups on a given  \cst-algebra $\sA$ forms a category; the objects are pairs $(\QG,\alpha)$, where $\QG$ is a compact quantum group and $\alpha$ is an action of $\QG$ on $\alg$, and the morphism from $(\QG_1,\alpha_1)$ to $(\QG_2,\alpha_2)$ is a morphism from $\QG_1$ to $\QG_2$, represented by a map $\gamma:\C(\QG_2) \to \C(\QG_1)$ which intertwines the respective actions: $\alpha_1= (\id_{\sA} \ot \gamma) \circ \alpha_2$. When the category of actions of compact quantum groups on $\sA$ admits a universal (final) object $(\QG_u,\alpha_u)$, we call $\QG_u$ the \emph{quantum symmetry group} of $\sA$. These concepts originate in the article \cite{WangSym}, where it was shown that the category described above admits a universal object if $\alg$ is finite-dimensional and commutative, and in general, even for finite-dimensional $\sA$, one needs to add extra restrictions to the category of actions, such as the preservation of a fixed faithful state on $\sA$, for the universal object to exist.

This approach led to the development of the theory of quantum symmetry groups of finite graphs (\cite{BanBic}), finite metric spaces (\cite{banmet}), and later to the introduction of the concept of quantum isometry groups of noncommutative manifolds (\cite{gos}).
For more information we refer to the article \cite{orth}, where the following definition is proposed.

\begin{definition}\label{orthfilt}
Let $\sA$ be a  \cst-algebra equipped with a faithful state $\omega$ and with a family $(V_i)_{i\in\Ind}$ of finite-dimensional subspaces of $\sA$ (with the index set $\Ind$ containing a distinguished element 0) satisfying the following conditions:
\begin{enumerate}
\item $V_0=\CC\I_{\sA}$,
\item for all $i,j\in\Ind,\:i\neq{j}$, $a\in{V_i}$ and $b\in{V_j}$ we have $\omega(a^*b)=0$,
\item the set $\Lin\bigl(\bigcup\limits_{i\in\Ind}V_i\bigr)$
is a dense $*$-subalgebra of $\sA$.
\end{enumerate}
If the above conditions are satisfied we say that the pair $(\omega,(V_i)_{i\in\Ind})$ defines an orthogonal filtration of $\sA$; sometimes abusing the notation we will omit $\omega$ and simply say that $(\sA,(V_i)_{i\in\Ind})$ is a \cst-algebra with an orthogonal filtration.

Finally, we say that a quantum group $\QG$ acts on $\sA$ in a filtration preserving way if there exists an action $\alpha$ of $\QG$ on $\sA$ such that the following condition holds:
\[
\alpha(V_i)\subset{V_i}\odot\C(\QG),\qquad{i\in\Ind}.
\]
We write then $(\alpha,\QG)\in\catV$ (and call $\catV$ the corresponding category). 
\end{definition}

We have the following result.

\begin{theorem}[{\cite[Theorem 2.7]{orth}}] \label{qsymV}
Let $(\sA,\omega,(V_i)_{i\in \Ind})$ be a \cst-algebra with an orthogonal filtration. Then there exists a universal compact quantum group $\QG_u$ acting on $\sA$ in a filtration preserving way (i.e.~for any other compact quantum group $\QG$ acting on $\sA$ in a filtration preserving way there exists a unique morphism from $\QG$ to $\QG_u$ intertwining the respective actions). We call $\QG_u$ the quantum symmetry group of $(\sA,\omega,(V_i)_{i\in\Ind})$. The canonical action of $\QG_u$ on $\sA$ is faithful, that is the subalgebra $R_{\alpha}(\QG)$ of $\C(\QG)$ generated by the elements of the form $(f_i \ot \id)\alpha(V_i)$ ($i \in \Ind$, $f_i \in V_i^*$) is equal to $\Pol(\QG)$.
\end{theorem}

Note that the above theorem implies that each compact quantum group acting faithfully on $\sA$ in a filtration preserving way is a quantum subgroup of $\QG_u$.  Recall also that, as mentioned in the introduction, it has recently been extended to the case of quantum symmetry groups of \cst-modules equipped with orthogonal filtrations in \cite{Manon}.

The most important example for us is the one studied in Section 4 of \cite{orth}, which we now describe. Let $\Gamma$ be a discrete group. It embeds naturally into the reduced group \cst-algebra of $\Gamma$, and the corresponding elements of this $C^*$-algebra  will still be denoted by the same symbols; in particular we identify the group ring $\bc[\Gamma]$ as a subalgebra of $\C^*_r(\Gamma)$ via $\bc[\Gamma]=\textup{span}\{\gamma:\gamma \in \Gamma\}.$ The canonical trace on $\C^*_r(\Gamma)$ is given by the continuous extension of the formula:
\[\tau(\gamma) = \left\{ \begin{array}{ccc} 1 & \textup{ if } & \gamma=e \\ 0 & \textup{ if } & \gamma\neq e \\ \end{array} \right.\]
We consider below partitions of $\Gamma$ into finite sets, always assuming that $\{e\}$ (where $e$ denotes the neutral element of $\Gamma$)  is one of the sets in the partition. The following lemma is straightforward.

\begin{lemma}[Lemma 4.1, \cite{orth}] \label{filtGamma}
If $\mathcal{F}=(F_i)_{i\in \Ind}$ is a partition of $\Gamma$ into finite sets and $V^{\mathcal{F}}_i:=\textup{span}\{\gamma:\gamma\in F_i\}\subset \C^*_r(\Gamma)$ ($i\in \Ind$), then
the pair $( \tau, (V^{\mathcal{F}}_i)_{i\in \Ind})$ defines an orthogonal filtration of $\C^*_r(\Gamma)$.
\end{lemma}

\begin{deft} \label{qsymF}
The quantum symmetry group of $(\C^*_r(\Gamma), \tau, (V^{\mathcal{F}}_i)_{i\in \Ind})$, defined according to Theorem \ref{qsymV}, will be called the quantum symmetry group of $\widehat{\Gamma}$ preserving the partition $\mathcal{F}$.
\end{deft}

The language used above is supposed to reflect the fact that we think of $\C^*_r(\Gamma)$ as the (reduced) algebra of continuous functions on the compact quantum group $\widehat{\Gamma}$, the dual of $\Gamma$.


If $\Gamma$ is a finitely generated group, with a fixed finite symmetric generating set $S$ (not containing the neutral element), then the word-length related to $S$ naturally induces a partition $\mathcal{F}_S$ of $\Gamma$. In this case the quantum symmetry group of $\widehat{\Gamma}$ preserving the partition $\mathcal{F}_S$ can be identified with the quantum isometry group of the spectral triple $(\bc[\Gamma], \ell^2(\Gamma),D_l)$, where $D_l$ is the Dirac operator related to the length function $l:\Gamma \to \NN_0$ (for the respective definitions and details we refer to the articles \cite{Deb2}, \cite{bhs}, \cite{Symmetric} and \cite{orth}). We will thus call the  quantum symmetry group of $\widehat{\Gamma}$ preserving the partition $\mathcal{F}_S$ the \emph{quantum isometry group of} $\widehat{\Gamma}$ and denote it
$\QI(\widehat{\Gamma}, S)$, or simply $\QI(\widehat{\Gamma})$, when the choice of the generating set is clear from the context.

\subsection{Inductive limits of \cst-algebras, projective limits of compact quantum groups}

For a short, but exhaustive description of the inductive limit construction for \cst-algebras we refer to Section II.8.2 of \cite{bl} or to Chapter 6 of \cite{Niels}. We will use it in the following context: let $(\alg_n)_{n\in \NN}$ be a sequence of  \cst-algebras and assume that for all $n,m\in \NN, n<m$, there exists a $*$-homomorphism $\pi_{n,m}:\alg_n \to \alg_m$ and that the compatibility conditions hold:
\[
\pi_{m,k}\comp\pi_{n,m}=\pi_{n,k},\qquad{n<m<k}
\]
(note that we do not assume that the connecting maps are injective).
Then there exists a unique (up to isomorphism)  \cst-algebra $\alg_{\infty}$ and  $*$-homomorphisms $\pi_{n,\infty}:\alg_n \to \alg_{\infty}$ such that
\[
\pi_{m,\infty}\comp\pi_{n,m}=\pi_{m,\infty},\qquad{n<m},
\]
and $\alg_{\infty}$ is generated by the union $\bigcup_{n\in \NN} \pi_{n,\infty}(\alg_n)$. The algebra $\alg_{\infty}$ satisfies the natural universal property: whenever $\alg$ is a \cst-algebra equipped with the $*$-homomorphisms $\gamma_{n}:\alg_n \to \alg$ `compatible' with the family $\{\pi_{n,m}:n,m\in \NN, n<m\}$, there exists a unique  $*$-homomorphism $\pi:\alg_{\infty}\to \alg$ such that $\pi\circ \pi_{n,\infty} = \gamma_n$ for all $n \in \NN$.

The following result was shown in \cite{Wangfree} (see also \cite{bgs}).

\begin{lemma} \label{indqg}
Suppose that $(\QG_n)_{n=1}^{\infty}$ is a sequence of compact quantum groups, that for each $n,m \in\NN,\;n<m$ there exists a compact quantum group morphism $\pi_{n,m}$ from $\QG_m$ to $\QG_n$ and the compatibility conditions
\[
\pi_{m,k}\comp\pi_{k,n}=\pi_{m,n},\qquad{n<k<m},
\]
hold. Then the inductive limit of the sequence $\bigl(\C(\QG_n)\bigr)_{n=1}^{\infty}$ of \cst-algebras admits a canonical structure of the algebra of continuous functions on a compact quantum group. Denote the resulting compact quantum group by $\QG_{\infty}$ and let for each $n\in\NN$ the associated morphism from $\QG_{\infty}$ to $\QG_n$ be denoted by $\pi_{n,\infty}$. Then $\QG_{\infty}$ has the following universal property: for any compact quantum group $\QH$ such that there exists a family of (compatible in a natural sense) morphisms $\gamma_n$ from $\QH$ to $\QG_n$ there exists a unique morphism $\gamma$ from $\QH$ to $\QG_{\infty}$ such that $ \gamma\comp\pi_{n,\infty}=\gamma_n$. We will sometimes write
\[
\QG_{\infty}=\varprojlim\QG_n.
\]
\end{lemma}

This construction should be thought of as a quantum version of the inverse limit construction for compact groups (see Section \ref{p-adic}).

\section{Doubling of Hopf $*$-algebras and its connections to quantum symmetry groups}\label{doubl}

In this section we present the doubling construction for Hopf $*$-algebras and explain its connection to quantum symmetry groups of duals of discrete groups.

\subsection{Doubling of Hopf $*$-algebras} \label{doublss}

The simple procedure of \emph{doubling} a Hopf algebra or, more generally, a regular multiplier Hopf algebra with invariant functionals, was described in \cite[Section 3]{Symmetric}. It needs to be emphasized that it is a different (in fact, far simpler) procedure than that of constructing the \emph{quantum double} of a Hopf algebra from \cite{Drinfeld}. The doubling construction takes as input data a Hopf algebra (or a regular multiplier Hopf algebra with invariant functionals) $(\cA,\Delta_\cA)$ equipped with a (multiplier) Hopf algebra automorphism $\theta\colon\cA\to\cA$ such that $\theta^2=\id$. The procedure itself is quite simple: the doubling $(\widetilde{\cA},\Delta_{\widetilde{\cA}})$ of $(\cA,\Delta_\cA)$ is the dual of the regular multiplier Hopf algebra with invariant functionals obtained as the crossed product of $(\hh{\cA},\Delta_{\hh{\cA}})$ by the dual $\hh{\theta}$ of the automorphism $\theta$ (see \cite[Subsection 3.1]{Symmetric}).

In this paper we will describe many examples of quantum symmetry (isometry and other) groups which happen to be doublings of usually quite simple Hopf algebras (or multiplier Hopf algebras). We would like to list here several simple facts about such doublings. For simplicity we will restrict attention to unital Hopf $*$-algebras (in particular we will be interested in so-called \emph{CQG algebras}, i.e.\ those Hopf $*$-algebras which arise as as $\Pol(\QG)$ for a compact quantum group $\QG$, see \cite{DiK}). However, the whole theory can be developed (with minor technical complications) for regular multiplier Hopf algebras with invariant functionals, as defined in \cite{Algebraic}.

Instead of following the steps of the construction of $(\widetilde{\cA},\Delta_{\widetilde{\cA}})$ as indicated above we will now give a direct description of this object leaving the proof that $(\widetilde{\cA},\Delta_{\widetilde{\cA}})$ is indeed the doubling of $(\cA,\Delta_\cA)$ to the reader. Let $(\cA,\Delta_\cA)$ be a Hopf $*$-algebra over $\CC$ and let $\theta$ be an order-two automorphism of $(\cA,\Delta_\cA)$. Then $\widetilde{\cA}$ is as a $*$-algebra isomorphic to $\cA\oplus\cA$. Let us denote the injections of $\cA$ onto the first and second coordinate in $\widetilde{\cA}$ by $\xi$ and $\eta$ respectively:
\[
\xi(a)=(a,0),\quad\eta(a)=(0,a),\qquad(a\in\cA).
\]
The comultiplication $\Delta_{\widetilde{\cA}}\colon\widetilde{\cA}\to\widetilde{\cA}\odot\widetilde{\cA}$ is determined by the properties:
\begin{equation}\label{DeltaATil}
\begin{split}
\Delta_{\widetilde{\cA}}\comp\xi&=\bigl(\xi\tens\xi+\eta\tens[\eta\comp\theta]\bigr)\comp\Delta_\cA,\\
\Delta_{\widetilde{\cA}}\comp\eta&=\bigl(\xi\tens\eta+\eta\tens[\xi\comp\theta]\bigr)\comp\Delta_\cA.
\end{split}
\end{equation}

Equipped with formulas \eqref{DeltaATil} one can easily show that if $h$ is a left Haar state (left invariant state) on $\cA$ then the functional
\begin{equation}\label{doublH}
\widetilde{\cA}\ni(a,b)\longmapsto\tfrac{1}{2}\bigl(h(a)+h(b)\bigr)
\end{equation}
is left invariant.

Once $(\widetilde{\cA},\Delta_{\widetilde{\cA}})$ is constructed one can observe that we have a coaction of this Hopf algebra on $\cA$ given by
\begin{equation}\label{coact}
\alpha=\bigl(\id\tens\xi+\theta\tens[\eta\comp\theta]\bigr)\comp\Delta_\cA\colon\cA\to\cA\odot\widetilde{\cA}.
\end{equation}
Moreover this coaction is \emph{embeddable} as defined in \cite[Definition 1.8]{podles}. Indeed, we have an injective $*$-homomorphism $\Xi\colon\cA\to\widetilde\cA$ defined as
\begin{equation}\label{emb}
\Xi(a)=\xi(a)+\eta\bigl(\theta(a)\bigr),\qquad(a\in\cA),
\end{equation}
which satisfies
\[
\Delta_{\widetilde{\cA}}\comp\Xi=(\Xi\tens\id)\comp\alpha.
\]

Consider two Hopf $*$-algebras $(\cA_1,\Delta_{\cA_1})$ and $(\cA_2,\Delta_{\cA_2})$ with order two automorphisms $\theta_1$ and $\theta_2$ and their respective doublings $(\widetilde{\cA}_1,\Delta_{\widetilde{\cA}_1})$ and $(\widetilde{\cA}_2,\Delta_{\widetilde{\cA}_2})$. For $i=1,2$ let $\xi_i,\eta_i\colon\cA_i\to\widetilde{\cA}_i$ be the injections defined above. They provide a convenient way to extend linear maps $\cA_1\to\cA_2$ to maps $\widetilde{\cA}_1\to\widetilde{\cA}_2$: given a linear map $f\colon\cA_1\to\cA_2$ we will denote by $\widetilde{f}$ the unique map $\widetilde{\cA}_1\to\widetilde{\cA}_2$ satisfying
\[
\widetilde{f}\comp\xi_1=\xi_2\comp{f},\qquad\widetilde{f}\comp\eta_1=\eta_2\comp{f}.
\]
In particular, taking $\cA_2=\cA_1$ and for $f$ the automorphism $\theta_1$ we obtain an automorphism $\widetilde{\theta}_1$ of $\widetilde{\cA}_1$. Similarly we obtain an automorphism $\widetilde{\theta}_2$ of $\widetilde{\cA}_2$.

The following properties of the extension $f\mapsto\widetilde{f}$ are obvious:
\begin{itemize}
\item an equivariant map (i.e.~$f$ such that $\theta_2\comp{f}=f\comp\theta_1$) has an equivariant extension: $\widetilde{\theta}_2\comp\widetilde{f}=\widetilde{f}\comp\widetilde{\theta}_1$,
\item if $f$ is injective/surjective then so is $\widetilde{f}$,
\item if $f$ is an algebra homomorphism then so is $\widetilde{f}$,
\item if $f$ is an equivariant Hopf algebra map then $\widetilde{f}$ is a Hopf algebra map,
\item if $(\cA_3,\Delta_{\cA_3})$ is another Hopf algebra with an order two automorphism $\theta_3$ and with doubling $(\widetilde{\cA}_3,\Delta_{\widetilde{\cA}_3})$ and if $g\colon\cA_2\to\cA_3$ is a linear map then $\widetilde{g\comp{f}}=\widetilde{g}\comp\widetilde{f}$.
\end{itemize}

For yet another property of the extension $f\mapsto\widetilde{f}$ let
\[
\begin{array}{r@{\:\colon\:}l}
\alpha_i&\cA_i\to\cA_i\odot\widetilde{\cA}_i,\\
\Xi_i&\cA_i\to\widetilde{\cA}_i
\end{array}
\qquad{i=1,2},
\]
be the coactions and embeddings as in \eqref{coact} and \eqref{emb}. Now if $f\colon\cA_1\to\cA_2$ is equivariant then
\[
\Xi_2\comp{f}=\widetilde{f}\comp\Xi_1,
\]
so if $f$ is also a Hopf algebra map then
\begin{equation}\label{Xialf}
\begin{split}
(\Xi_2\tens\id)\comp\alpha_2\comp{f}
&=\Delta_{\widetilde{\cA}_2}\comp\Xi_m\comp{f}=\Delta_{\widetilde{\cA}_2}\comp\widetilde{f}\comp\Xi_1\\
&=\bigl(\widetilde{f}\tens\widetilde{f}\,\bigr)\comp\Delta_{\widetilde{\cA}_1}\comp\Xi_1\\
&=(\Xi_2\tens\id)\comp\bigl(f\tens\widetilde{f}\,\bigr)\comp\alpha_1.
\end{split}
\end{equation}
Now let $\pi_2$ be the map
\[
\cA_2\ni(a,b)\longmapsto{a}\in\cA_2.
\]
Then $\pi_2\comp\Xi_2=\id_{\cA_2}$. Thus composing both sides of \eqref{Xialf} from the left with $(\pi_2\tens\id)$ we obtain
\[
\alpha_2\comp{f}=\bigl(f\tens\widetilde{f}\,\bigr)\comp\alpha_1.
\]

We will use these properties in Section \ref{symm}.

The last result in this section is related to the question when two given doublings of a given Hopf $*$-algebra $\cA$, say given respectively by order two-automorphisms $\theta_1$ and $\theta_2$ of $\cA$, are isomorphic as Hopf $*$-algebras. It is easy to verify (for example using the formulas above) that if $\theta_1$ and $\theta_2$ are \emph{conjugate}, i.e.\ there exists an automorphism $\theta$ of $\cA$ such that $\theta \circ \theta_1 = \theta_2 \circ \theta$, then the doublings
$\widetilde{\cA}_1$ and $\widetilde{\cA}_2$, associated respectively to $\theta_1$ and $\theta_2$, are isomorphic. It turns out that in some cases, which will turn out to be important for our considerations later in the paper, this is the only way in which an isomorphism of two different doublings of a Hopf $*$-algebra $\cA$ can arise. This is the content of the next proposition.

\begin{proposition}\label{isomDoubl}
Suppose that $\Gamma$ is a non-trivial ICC  group (a group with infinite conjugacy classes) and let $\cA=\bc[\Gamma]$. Suppose that $\theta_1,\theta_2$ are (Hopf $*$-algebra) automorphisms of $\cA$ and that the respective doublings induced by $\theta_1$ and $\theta_2$,  $\widetilde{\cA}_1$ and $\widetilde{\cA}_2$, are isomorphic. Then the automorphisms $\theta_1$ and $\theta_2$ are conjugate.
\end{proposition}

\begin{proof}
Recall first that $\widetilde{\cA}_1$ and $\widetilde{\cA}_2$ are  isomorphic as $*$-algebras to the direct sum $\cA \oplus \cA$. Let $\theta_1$ and $\theta_2$ be as above and denote the respective coproducts  of $\widetilde{\cA}_1$ and $\widetilde{\cA}_2$ by $\Com_1$ and $\Com_2$. Let $\phi$ be an automorphism of the $*$-algebra $\cA \oplus \cA$  such that
\begin{equation} (\phi \ot \phi)\circ \Com_1 = \Com_2 \circ \phi.\label{phiint}\end{equation}
As $\Gamma$ is assumed to be ICC, it is well-known and easy to check that the centre of the group-ring $\bc[\Gamma]$ consists only of the scalar multiples of identity. This implies that the only non-trivial (i.e.\ non-zero and not equal to $1$) central projections in the algebra $\cA\oplus \cA$ are $p_1 =\xi(1_{\cA}), p_2 =\eta(1_{\cA})$ (note we are using here the notation introduced earlier for two algebra embeddings of $\cA$ into $\cA \oplus \cA$). Thus we must have either $\phi(p_1)=p_2$ or $\phi(p_1)=p_1$. Suppose for the moment the first case holds. Then $\phi(p_2)=p_1$ and as the equation \eqref{DeltaATil} implies that
$\Com_1 (p_1)  = p_1 \ot p_1 + p_2 \ot p_2$ and $\Com_2(p_2) = p_1 \ot p_2 + p_2 \ot p_1$, applying the  formula \eqref{phiint} to $p_1$ yields
\[ p_2 \ot p_2 + p_1 \ot p_1 = p_2 \ot p_1 + p_1 \ot p_2,\]
which is contradictory. Thus we must have $\phi(p_1)=p_1$, $\phi(p_2)=p_2$. This is easily seen to imply that there exist two automorphisms of the algebra $\cA$, which we will denote respectively by $\psi$ and $\rho$, such that
\[ \phi \circ \xi = \xi \circ \psi, \;\;\; \phi \circ \eta = \eta \circ \rho.\]
Substituting the last formulas into \eqref{phiint} and using \eqref{DeltaATil} again yields the following four equalities, where now $\Com$ denotes the coproduct of $\cA$:
\begin{equation} (\psi \ot \psi) \circ \Com = \Com \circ \psi, \;\;\; (\rho \ot (\rho \circ \theta_1))\circ \Com = (\id \ot \theta_2) \circ \Com \circ \psi,\label{foureq1}\end{equation}
\begin{equation}  (\psi \ot \rho) \circ \Com = \Com \circ \rho, \;\;\; (\rho \ot (\psi \circ \theta_1))\circ \Com = (\id \ot \theta_2) \circ \Com \circ \rho. \label{foureq2}\end{equation}
Applying the counit to the right leg of the first equality in \eqref{foureq2} yields
\[ \rho = \psi \circ T,\]
where $T:=(\id \ot ( \epsilon \circ \rho))\circ \Com:\cA \to \cA$. The map $T$ is an automorphism of $\cA$. Moreover, for any $\gamma \in \Gamma$ we have
\[T(\gamma) = \gamma \lambda(\gamma),\]
where $\lambda(\gamma):=\epsilon\circ \rho(\gamma)$ defines a character of $\Gamma$. Apply then the second formula in \eqref{foureq2} to a $\gamma\in \Gamma$, obtaining
\[ \psi(\gamma)\lambda(\gamma) \ot \psi(\theta_1(\gamma)) = ((\id \ot \theta_2) \circ \Com) ( \psi(\gamma) \lambda(\gamma)).\]
Using the first equality in \eqref{foureq1} yields now
\[ \psi(\gamma)\lambda(\gamma) \ot \psi(\theta_1(\gamma)) = \lambda(\gamma) \psi(\gamma) \ot \theta_2(\psi(\gamma)).\]
As $\lambda$ is a character, so cannot be equal to $0$ at any point, we obtain
\[ \psi(\theta_1(\gamma)) = \theta_2(\psi(\gamma)),\]
so further by arbitrariness of $\gamma$ and linear extension we see that $\psi \circ \theta_1 = \theta_2 \circ\psi$ and putting $\theta:=\psi$ ends the proof.
\end{proof}

Note that the Hopf $*$-algebra automorphisms of $\bc[\Gamma]$ are necessarily induced by automorphisms of $\Gamma$, and in particular their conjugacy is equivalent to the conjugacy of the underlying group automorphisms.

\subsection{Doubling of compact quantum groups and connections to quantum symmetry groups}

It is easy to check,  for example using the explicit formulas for the Haar state on $\widetilde{\cA}$, that if $\cA$ is a CQG algebra, then so is $\widetilde{\cA}$. Thus we can exploit the construction above as the basis of the following definition.

\begin{definition}
Let $\QG$ be a compact quantum group and suppose we are given an order two automorphism of $\QG$, say represented by a map $\theta:\C(\QG)\to \C(\QG)$. Then we define the compact quantum group $\QG_{\theta}$ by the equality $\Pol(\QG_{\theta}):=\widetilde{\Pol(\QG)}$, where the doubling construction on the right side of the equation refers to the Hopf $*$-algebra morphism $\theta|_{\Pol(\QG)}$.
\end{definition}

We leave it to the reader to verify that if $\QG$ is in fact a compact group (in other words, $\C(\QG)$ is commutative), then $\QG_{\theta}$ is also a compact group and we have simply $G_{\theta} = G \rtimes_{\theta} \bz/_{2\bz}$.

The arguments above show that $\QG_{\theta}$ always acts on $\C(\QG)$. We will now consider a special case, which is closely related to the fact that the doubling construction often appears in the computation of quantum isometry groups of duals of finitely generated groups  (see \cite{bhs}, \cite{Symmetric} and \cite{Chin}).

\begin{proposition} \label{doubsym}
Let $\Gamma$ be a finitely generated discrete group, with a fixed finite generating set $S$. Assume that $\theta:\Gamma \to \Gamma$ is an order two automorphism leaving the set $S$ invariant. Denote the induced automorphism of the Hopf $*$-algebra $\bc[\Gamma]$ by the same symbol. Then the quantum group $\widehat{\Gamma}_{\theta}$ (recall that we define the compact quantum group $\widehat{\Gamma}$ via the identification $\Pol(\widehat{\Gamma})\approx \bc[\Gamma]$) acts on $\C*_r(\Gamma)$, preserving the partition $\mathcal{F}_S$ (see the last paragraph of Subsection \ref{prelCQG}). If moreover $\theta$ is not equal to identity then the corresponding action is faithful, so $\widehat{\Gamma}_{\theta}$ is a quantum subgroup of $\QI(\widehat{\Gamma}, S)$.
\end{proposition}

\begin{proof}
It is easy to see that $\theta$ induces an order two automorphism of $\bc[\Gamma]$ (given simply by the linear extension of the original map). As the resulting automorphism  preserves the canonical trace, it extends also to an automorphism of the reduced group \cst-algebra $\C^*_r(\Gamma)$. Let then $\widehat{\Gamma}_{\theta}$ denote the compact quantum group arising via the doubling construction, so that $\C(\widehat{\Gamma}_{\theta})= \C^*_r(\Gamma) \oplus \C^*_r(\Gamma)$. The formula \eqref{coact} implies that the action of $\widehat{\Gamma}_{\theta}$ on $\C^*_r(\Gamma)$ is given by the expressions:
\begin{equation} \alpha(\gamma) = \gamma \ot (\gamma, 0) + \theta(\gamma) \ot (0,\theta(\gamma)), \;\;\; \gamma \in \Gamma.\label{actformula}\end{equation}
Thus the action $\alpha$ preserves the orthogonal filtration of $\C^*_r(\Gamma)$ induced by the word-length partition if and only if $\theta$ preserves the length function. It is easy to see that the assumptions of the proposition imply the latter fact.

By the remark stated after Theorem \ref{qsymV} it remains to show that if $\theta$ is not constant on $S$, then the action described in the first part of the proof is faithful. So let $s\in S$ be such that $\theta(s)\neq s$. Choosing a suitable functional on the subspace $\Lin S\subset \C^*_r(\Gamma)$, we can assure that $R_{\alpha}(\widehat{\Gamma}_{\theta})$  (introduced in the paragraph after Definition \ref{actdef}) contains elements of the form $(s,0)$, $(s^{-1}, 0)$, $(0,s), (0,s^{-1})$. On the other hand if $t\in S$ and $\theta(t)=t$, then similarly
the elements of the form $(t,t)$ and $(t^{-1}, t^{-1})$ belong to $R_{\alpha}(\widehat{\Gamma}_{\theta})$. It is easy to check that considering in the same way all elements of $S$ one can show that $R_{\alpha}(\widehat{\Gamma}_{\theta}) =\bc[\Gamma]\oplus \bc[\Gamma]$.

\end{proof}

Note that the assumption that $\theta$ is non-trivial is indeed necessary to ensure that the action of $\widehat{\Gamma}_{\theta}$  on $\C^*_r(\Gamma)$ is faithful. In several cases it turns out that $\widehat{\Gamma}_{\theta} = \QI(\widehat{\Gamma}, S)$; we will see some instances of this in the following sections.


\section{Generalities on limits of quantum symmetry groups}\label{gensect}

In this section we study certain projective limits of quantum symmetry groups. We begin with recalling a theorem proved in \cite{bgs}. For the precise terminology used below we refer to the articles \cite{bgs} and \cite{Deb2}.

\begin{theorem}[Theorem 1.2 of \cite{bgs}] \label{GoodDirac}
Suppose that $\sA$ is a \cst-algebra acting on a Hilbert space $\sH$ and that $D$ is a (densely defined) selfadjoint operator on $\sH$ with compact resolvent, such that $D$ has a one-dimensional eigenspace spanned by a vector $\xi$ which is cyclic and separating for $\sA$. Let $(\Alg_n)_{n\in\NN}$ be an increasing net of a unital $*$-subalgebras of $\sA$ and put $\Alg=\bigcup\limits_{n\in\NN}\Alg_n$. Suppose that $\Alg$ is dense in $\sA$ and that for each $a\in\sA$ the commutator $[D,a]$ is densely defined and bounded. Additionally put $\sH_n=\Bar{\Alg_n\xi}$, let $P_n$ denote the orthogonal projection onto $\sH_n$ and assume that each $P_n\sH$ is a direct sum of eigenspaces of $D$. Then each $(\Alg_n,\sH_n,\bigl.D\bigr|_{\sH_n})$ is a spectral triple satisfying the conditions of Theorem 2.14 of \cite{Deb2}, there exist natural compatible compact quantum group morphisms $\pi_{n,m}\colon\C\bigl(\QI(\Alg_n,\sH_n,\bigl.D\bigr|_{\sH_n})\bigr)\to\C\bigl(\QI(\Alg_m,
 \sH_m,\bigl.D\bigr|_{\sH_m})\bigr)\;(n,m\in\NN,\:n<m)$ and
\[
\QI(\sA,\sH,D)=\varprojlim\QI(\Alg_n,\sH_n,\bigl.D\bigr|_{\sH_n}).
\]
\end{theorem}

\begin{remark}
In fact the statement of the corresponding theorem in \cite{bgs} is incorrect: the assumption that $D$ commutes with the projections $P_n$ does not imply the existence of the quantum group morphisms from $\QI(\Alg_m,\sH_m,\bigl.D\bigr|_{\sH_m})$ to $\QI(\Alg_n,\sH_n,\bigl.D\bigr|_{\sH_n})$. 
Note, however, that in all the applications studied in \cite{bgs} the assumptions of the above correct formulation are satisfied. In fact we will show below another, more general result, which uses the language of quantum symmetry groups of orthogonal filtrations, as developed in \cite{orth}.
\end{remark}


As we will see in the last section, there are examples in which the assumptions of the above theorem are not satisfied and yet certain quantum symmetry groups form natural  projective systems. One can expect that in such situations one could also produce the action of the quantum group arising as the limit on a quantum space and possibly also identify it as a quantum symmetry group of some structure. The first step for that is provided by the following proposition.

\begin{proposition}\label{indactQG}
Suppose that $(\QG_n)_{n=1}^{\infty}$ is a sequence of compact quantum groups with compact quantum group morphisms $\pi_{n,m}$ from $\QG_m$ to $\QG_n$ satisfying the compatibility conditions as above. Let $(\sA_n)_{n=1}^{\infty}$ be a sequence of unital \cst-algebras, with connecting unital $*$-homomorphisms $\rho_{n,m}\colon\sA_n\to\sA_m$, again satisfying the natural compatibility conditions, and let $\sA_{\infty}$ denote the resulting inductive limit. Assume that for each $n\in\NN$ there is an action $\alpha_n$ of $\QG_n$ on $\sA_n$ (so that $\alpha_n\colon\sA_n\to\sA_n\tens\C(\QG_n)$) and that for all $n,m\in\NN,\:n<m$ we have
\begin{equation}\label{intertw}
\alpha_m\comp\rho_{n,m}=(\rho_{n,m}\tens\pi_{n,m})\comp\alpha_n.
\end{equation}
Then $\QG_{\infty}=\varprojlim\QG_n$ acts on $\sA_{\infty}$ by an action $\alpha_{\infty}\colon\sA_{\infty}\to\sA_{\infty}\tens\C(\QG_{\infty})$ uniquely determined by the set of conditions:
\[
\alpha_{\infty}\comp\rho_{n,\infty}=(\rho_{n,\infty}\tens\pi_{n,\infty})\comp\alpha_n,\qquad{n\in\NN}.
\]
Moreover if the connecting morphisms $\rho_{n,m}$ are injective and each of the actions $\alpha_n$ is faithful, the limit action $\alpha_{\infty}$ is also faithful.
\end{proposition}

\begin{proof}
The result is a straightforward consequence of the universal property of the inductive limit. Indeed, for each $n\in\NN$ we can define a unital $*$-homomorphism $\beta_n\colon\sA_n\to\sA_{\infty}\tens\C(\QG_{\infty})$ by the formula
\[
\beta_n=(\rho_{n,\infty}\tens\pi_{n,\infty})\comp\alpha_n,
\]
where $\pi_{n,\infty}\colon\C(\QG_n)\to\C(\QG_{\infty})$ are the canonical quantum group morphisms. It is easy to check that for all $n,m\in\NN,\:n<m$,
\[
\beta_m\comp\rho_{n,m}=\beta_n,
\]
so that they induce a unital $*$-homomorphism $\alpha_{\infty}\colon\sA_{\infty}\to\sA_{\infty}\tens\C(\QG_{\infty})$. It is also easy to check that it satisfies 
\end{proof}

In the remaining part of this section we will formulate and prove a generalisation of Theorem \ref{GoodDirac} to the context of quantum symmetry groups associated with orthogonal filtrations.
Suppose now again that $(\sA_n)_{n=1}^{\infty}$ is an inductive system of unital \cst-algebras, with the connecting unital $*$-homomorphisms $\rho_{n,m}\colon\sA_n\to\sA_m$. Assume that each $\sA_n$ is equipped with a faithful state $\phi_n$, an orthogonal filtration $\mathcal{V}_n$ (with respect to $\phi_n$) and that we have the following compatibility condition: for each $n,m\in\NN,\:n<m$, and each $V\in\mathcal{V}_n$ we have $\rho_{n,m}(V)\in\mathcal{V}_m$. Note that this implies the following two facts:
\begin{enumerate}
\item for all $n,m$ as above we have $\phi_m\comp\rho_{n,m}=\phi_m$,
\item each of the maps $\rho_{n,m}$ is injective (and thus so are the maps $\rho_{n,\infty}$).
\end{enumerate}
If the above assumptions hold, we will say that $(\sA_n,\mathcal{V}_n)_{n=1}^{\infty}$ is \emph{an inductive system of \cst-algebras equipped with orthogonal filtrations}. We can then speak about a natural inductive limit filtration $\mathcal{V}_{\infty}$ of the limit algebra $\sA_{\infty}$, defined in the following way: a subspace $V\subset\sA_{\infty}$ belongs to $\mathcal{V}_{\infty}$ if and only if there exists $n\in\NN$ and $V_n\in\mathcal{V}_n$ such that $V=\rho_{n,\infty}(V_n)$. The arising filtration satisfies then the orthogonality conditions with respect to the inductive limit state $\phi_{\infty}\in\sA_{\infty}^*$.

Note that there is one subtlety here: although we can always construct the inductive limit filtration, the inductive limit state need not be faithful on $\sA_{\infty}$, so we need not be in the framework studied in \cite{orth} and described in Definition \ref{orthfilt} -- we only know that $\phi_{\infty}$ is faithful on the dense subalgebra $\bigcup\limits_{n\in\NN}\bigcup\limits_{V\in\mathcal{V}_n}\rho_{n,\infty}(V)$ of $\sA_{\infty}$. Suppose in addition that $\phi_{\infty}$ is a trace (equivalently, each of the states $\phi_n$ is tracial). Then it is automatically faithful (as its null space, $\{a\in \sA_{\infty}:\phi_{\infty}(a^*a)=0\}$, is an ideal, cf.~\cite[Proposition II.8.2.4]{bl}). We can thus formulate and prove the following theorem.

\begin{theorem}\label{mainlimit}
Let $(\sA_n,\mathcal{V}_n)_{n=1}^{\infty}$ be an inductive system of \cst-algebras equipped with orthogonal filtrations and assume that each of the states defining the orthogonality of the filtrations $\mathcal{V}_n$ is tracial. Let $\mathcal{V}_{\infty}$ denote the orthogonal filtration of $\sA_{\infty}$ arising as the inductive limit. Denote the respective quantum symmetry groups of $(\sA_n,\mathcal{V}_n)$ and $(\sA_{\infty},\mathcal{V}_{\infty})$ by $\QG_n$ and $\QG$. Then
\[
\QG =\varprojlim\QG_n.
\]
\end{theorem}

\begin{proof}
Begin by showing that the expression on the right hand side of the displayed formula above makes sense, i.e.~that for all $m,n\in\NN,\:m>n$, there exist natural quantum group morphisms from $\QG_m$ to $\QG_n$. Recall that $\QG_n$ is defined as the universal object in the category $\mathcal{C}_{\sA_n,\mathcal{V}_n}$ (the corresponding action will be denoted by $\alpha_n$, so that $\alpha_n\colon\sA_n\to\sA_n\tens\C(\QG_n)$). Therefore it effectively suffices to show that $\QG_m$ acts on $\sA_n$ in a filtration preserving fashion. Fix $m>n$ and consider the map $\bigl.\alpha_m\bigr|_{\rho_{n,m}(\sA_n)}$. As $\alpha_m$ preserves the filtration $\mathcal{V}_m$, it in particular preserves the linear span of the set $\{\rho_{n,m}(V):V\in \mathcal{V}_n\}$. But the latter is a dense unital $*$-subalgebra of $\rho_{n,m}(\sA_n)$, so we have
\[
\alpha_m\bigl(\rho_{n,m}(\sA_n)\bigr)\subset\rho_{n,m}(\sA_n)\tens\C(\QG_m).
\]
Thus we can define $\beta_{n,m}\colon\sA_n\tens\sA_n\tens\C(\QG_m)$ by the formula
\begin{equation}\label{betanm1}
\beta_{n,m} =({\rho_{n,m}}^{-1}\tens\id_{\C(\QG_m)})\comp\alpha_m\comp\rho_{n,m}.
\end{equation}
It is easy to check that in fact $\beta_{n,m}$ defines an action of $\QG_m$ on $\sA_n$ preserving the filtration $\mathcal{V}_n$. Thus there exists a unique quantum group morphism $\pi_{n,m}\colon\C(\QG_n)\to\C(\QG_m)$ such that
\begin{equation}\label{betanm2}
(\id_{\sA_n}\tens\pi_{n,m})\comp\alpha_n=\beta_{n,m}.
\end{equation}
Comparison of \eqref{betanm1} and \eqref{betanm2} implies that $\pi_{n,m}\colon\C(\QG_n)\to\C(\QG_m)$ is the unique $*$-homomorphism such that
\[
(\rho_{n,m}\tens\id_{\C(\QG_m)})\comp(\id_{\sA_n}\tens\pi_{n,m})\comp\alpha_n=\alpha_m\comp\rho_{n,m}.
\]
The uniqueness in the above formula implies that $\pi_{m,k}\comp\pi_{n,m}=\pi_{n,k}$ for all $n,m,k\in\NN,\:n<m<k$. Denote the limit quantum group for the system $(\QG_n)_{n=1}^{\infty}$ (with the connecting maps $\pi_{n,m}$) by $\QG_{\infty}$. Proposition \ref{indactQG} implies that $\QG_{\infty}$ acts on $\sA_{\infty}$ by the action $\alpha_{\infty}$ such that
\[
\alpha_{\infty}\comp\rho_{n,\infty}=(\rho_{n,\infty}\tens\pi_{n,\infty})\comp\alpha_n,\qquad{n\in\NN}.
\]
It is easy to check that the action $\alpha_{\infty}$ preserves the filtration $\mathcal{V}_{\infty}$.

It remains to show that the pair $(\QG_{\infty},\alpha_{\infty})$ is the universal object in the category  $\mathcal{C}_{\sA_{\infty}, \mathcal{V}_{\infty}}$. Assume than that $(\beta,\QH)$ belongs to $\mathcal{C}_{\sA_{\infty},\mathcal{V}_{\infty}}$. This means that $\beta\colon\sA_{\infty}\to\sA_{\infty}\tens\C(\QH)$ is an action which preserves the filtration $\mathcal{V}_{\infty}$. Exactly as in the first part of the proof, this implies that $\beta\bigl(\rho_{n,\infty}(\sA_n)\bigr)\subset\rho_{n,\infty}(\sA_n)\tens\C(\QH)$. Construct an action of $\QH$ on $\sA_n$ by the formula
\begin{equation}\label{bnm1}
\beta_{n,\infty}=({\rho_{n,\infty}}^{-1}\tens\id_{\C(\QH)})\comp\beta\comp\rho_{n,\infty}.
\end{equation}
As before, $\beta_{n,\infty}$ preserves the filtration $\mathcal{V}_n$, so that there exists a unique quantum group morphism $\gamma_{n}\colon\C(\QG_n)\to\C(\QH)$ such that
\begin{equation}\label{bnm2}
(\id_{\sA_n}\tens\gamma_{n})\comp\alpha_n=\beta_{n,\infty}.
\end{equation}
Easily checked compatibility properties and the universal property of $\QG_{\infty}$ imply existence of a unique $*$-homomorphism $\gamma\colon\C(\QG_{\infty})\to\C(\QH)$ such that $\gamma_n=\gamma\comp\pi_{n,\infty},\;(n\in\NN)$. It is easy to check that $\gamma$ intertwines respective actions of $\QH$ and $\QG_{\infty}$ on $\sA_{\infty}$. This ends the proof.
\end{proof}

It follows from \cite[Theorem 2.10]{orth} that the compact quantum group constructed as the limit in the above theorem is necessarily of Kac type. Note that although we assume that the connecting morphisms in the inductive system $(\sA_n)_{n=1}^{\infty}$ are injective, the connecting morphisms $\C(\QG_n)\to\C(\QG_{n+1})$ in general need not be injective (see Section \ref{p-adic}).

\section{Quantum symmetry groups associated with the algebras of continuous functions on the group of $p$-adic numbers} \label{p-adic}

In this section we present an elementary example of a construction from Theorem \ref{mainlimit} in the context of quantum symmetry groups acting on quantum group duals in a length preserving fashion (\cite{bhs}, see also Subsection 1.1 above). It will turn out that the quantum symmetry groups arising in the respective limit are usual compact groups.

Choose a prime number $p\in \NN$ (in fact it is not important that $p$ is a prime, we choose it so rather for  traditional reasons). Recall the construction
of p-adic integers: these form an abelian locally compact group arising as the inverse limit of the cyclic groups $(\Zpn)_{n=1}^{\infty}$ and will be denoted by $\Pad$ (note that this notation is not compatible with that used in \cite{bhs}!). The connecting surjective homomorphisms $j_n$ from $\Zpnplus$ to $\Zpn$ are given by the `reduction modulo $p^n$'. To simplify the notation write $\Gamma_l$ for $\widehat{\Zl}$ ($l\in \NN$). It is easy to check that if we denote the canonical generators of $\Gamma_l$ by $e_l$ (using a standard Fourier identification of $\Gamma_l$ with a cyclic abelian group of order $l$), then the dual (injective) homomorphism $\widehat{j_n}: \Gamma_{p^n} \to \Gamma_{p^{n+1}} $ is given by the formula
\[ \widehat{j_n} (e_{{p^n}}) = p e_{p^{n+1}}.\]
Denote the  elements of the group \cst-algebra $\C^*(\Gamma_l)$  corresponding to $t e_l$ by $\stackrel{l}{\lambda}_t$ ($t=0,1,\ldots,l-1$) and let the injective morphism corresponding to $\widehat{j_n}$ on the level of group \cst-algebras be denoted by $\iota_n$, so that $\iota_n: \C^*(\Gamma_{p^n})\to \C^*(\Gamma_{p^{n+1}})$ is given by
\begin{equation} \label{iotas} \iota_n (\stackrel{p^n}{\lambda}_1) = \stackrel{p^{n+1}}{\lambda}_p.\end{equation}

The next result is effectively just a rephrasing of the definition of $\Pad$.

\begin{proposition} \label{limitPad}
The algebra of continuous functions on $\Pad$ is naturally isomorphic (also as the algebra of functions on a compact quantum group)
to the inductive limit of the group \cst-algebras $\C^*(\Gamma_{p^n})$, with the connecting  morphisms given by the formula \eqref{iotas}.
\end{proposition}

\begin{remark}
It follows from the construction that $\C^*(\Pad)$ is isomorphic to the algebra of continuous functions on the Cantor set, $\C(\mathfrak{C})$; note however that the quantum symmetry group we construct below is different from the one acting on $\C(\mathfrak{C})$ considered in Theorem 3.1 of \cite{bgs}.
\end{remark}

The groups $\Zpn$ are of course finitely generated discrete groups and as such fit into the framework studied in \cite{bhs}. In each case the generating set will be chosen to be $\{e_l,-e_l\}$ and we will consider the partition of  $\Gamma_l$ into sets of equal word-lengths. Note the following key fact: if two elements $\gamma, \gamma' \in \Gamma_{p^n}$ have equal word-length, then so do $\widehat{j_n}(\gamma),
\widehat{j_n}(\gamma')\in \Gamma_{p^{n+1}}$. This implies the following extension of Proposition \ref{limitPad}.

\begin{proposition} \label{limitPadfilt}
Consider for each $n\in \NN$ the orthogonal filtration $\mathcal{V}_n$ of $\C^*(\Gamma_{p^n})$ induced by the partition of $\Gamma_{p^n}$ into sets of equal word-length.
Then $(\C^*(\Gamma_{p^n}), \mathcal{V}_n)_{n=1}^{\infty}$, with connecting maps defined in \eqref{iotas} is an inductive system of \cst-algebras equipped with orthogonal filtrations.
\end{proposition}

Thus we can ask about the quantum symmetry group of the limit orthogonal filtration $\mathcal{V}_{\infty}$ of $\C(\Pad)$. For that we need first to recall the form taken by the quantum symmetry (or isometry) groups at each stage of the construction. This was established in \cite{bhs}. Below we reformulate Theorem 3.1 of that paper in our language. Note that the actions of a classical group $G$ on the algebra $\C^*(\Gamma_l)\approx \C(\bz/_{l\bz})$ correspond to actions of $G$ on the set $\bz/_{l\bz}$.

\begin{theorem}
\label{finsym}
Let $l\in \NN$, $l\neq 1,2,4$. Then the quantum isometry group $\QI(\widehat{\Gamma_l})$, where $\Gamma_l$ is equipped with the generating set $\{e_l,-e_l\}$
is the (classical) group $G_l= \bz/_{l\bz} \rtimes \bz/_{2\bz}$, with the action of $\bz/_{2\bz}$  on $\bz/_{l\bz}$ given by the inverse automorphism ($t\mapsto t^{-1}$). Moreover the corresponding action of $G_l$ on the set $\bz/_{l\bz}$ is given by the formulas:
\[ (k,0) \cdot t = k+t, \;\;\; (k,1) \cdot t = k-t, \;\; k,t=0,\ldots,l-1.\]
\end{theorem}
\begin{proof}
It is shown in Theorem 3.1 of \cite{bhs} that $G_l$
given by the following formulas:
\[ \C(G_l) = \C^*(\bz/_{l\bz}) \oplus \C^*(\bz/_{l\bz}), \]
\[ \Com(A_l) = A_l \ot A_l + B_l \ot B_l^*, \;\; \Com(B_l) = A_l \ot B_l + B_l \ot A_l^*,\]
where we write $A_l= \stackrel{l}{\lambda}_1 \oplus 0$, $B _l =0 \oplus\stackrel{l}{\lambda}_1$.
The action of $G_l$ on $\C^*(\bz/_{l\bz})$ is shown to be given by the prescription:
\begin{equation}
\alpha_l (\stackrel{l}{\lambda}_1) = \stackrel{l}{\lambda}_1 \ot A_l + \stackrel{l}{\lambda}_{l-1} \ot B_l^*.  \end{equation}
The fact that these formulas can be rephrased in the way given in the statement of our theorem follows from an elementary calculation.
\end{proof}

Note that we view above the special case of the doubling construction introduced in Section 2.

Not surprisingly, the limiting quantum symmetry group associated with the limit orthogonal filtration $\mathcal{V}_{\infty}$ of $\C(\Pad)$ also has an analogous form. The restriction of $l>5$ in Theorem \ref{finsym} is here irrelevant, as we only ask about the limit behaviour.

\begin{theorem}
Let $\mathcal{V}_{\infty}$ denote the orthogonal filtration of $\C(\Pad)$ arising via Propositions \ref{limitPad} and \ref{limitPadfilt} and the construction described in Section \ref{gensect}. Then the corresponding quantum symmetry group $G$ is the group $ \Pad \rtimes \bz/_{2\bz}$, with the action of $\bz/_{2\bz}$  on $\Pad$ given by the inverse automorphism ($t\mapsto t^{-1}$). Moreover  the corresponding action of $G$ on the set $\Pad$ is given by the expected formulas:
\[ (k,0) \cdot t = k+t, \;\;\; (k,1) \cdot t = k-t, \;\; k,t \in \Pad.\]
\end{theorem}

\begin{proof}
Follows directly from Theorem \ref{mainlimit}, Theorem \ref{finsym} and some elementary commutative diagram chasing (one needs to verify that all considered actions are compatible with the connecting morphisms of the inductive system under consideration).
\end{proof}

Observe that if $l=4$ the corresponding quantum symmetry group is the free wreath product $\bz_2 \wr_* \bz_2$ (Theorem 3.2 of \cite{bhs}, Proposition  4.6 of \cite{orth}). Of course we cannot have an injective embedding of $\C(\bz_2 \wr_* \bz_2)$ into $\C(\bz_2 \rtimes \bz/_{2\bz})$, as the latter algebra  is commutative, and the former is not. This shows, as mentioned before, that in Theorem \ref{mainlimit} the connecting morphisms from $\C(\QG_n)$ to $\C(\QG_{n+1})$ need not be injective.

\section{Limits of quantum isometry groups of symmetric groups}\label{symm}

In the previous section we discussed an example of a sequence of finitely generated groups leading directly to the situation studied in Section \ref{doubl}: the corresponding group \cst-algebras with the filtrations induced by the word-length formed an inductive limit of \cst-algebras with orthogonal filtrations. This was a consequence of the fact that the  connecting morphisms acting from $\Gamma_{p^n}$ to $\Gamma_{p^{n+1}}$ preserved the level sets of the word-length functions in the suitable sense. In this section we present a different situation, arising in the context of quantum isometry groups of symmetric groups. The latter were computed in the article \cite{Symmetric}. The symmetric groups admit of course natural embeddings $S_n\hookrightarrow S_{n+1}$, but these turn out not to behave well with respect to the word-length induced by the generating sets built of transpositions and thus do not fit into the general framework of Section 2. Nevertheless we will show that one can still study (separately!) the projective limits of the corresponding sequences of quantum isometry groups indexed by even and odd integers, that these limits turn out to be non-isomorphic and moreover can be interpreted as the quantum symmetry groups of $\widehat{S_{\infty}}$ preserving  certain  partitions of $S_{\infty}$.

\subsection{The quantum isometry group of $\sS_n$}\label{qisoSn}

Let us begin by giving an account of the results of \cite[Section 2]{Symmetric} in the language used in the current article. Fix for the moment $n\in \NN\, (n\geq 3)$. Consider the symmetric group on $n$-elements, $\sS_n$, with the generating set
\begin{equation}\label{gener}
\bigl\{\sn_1,\sn_2,\dotsc,\sn_{n-1}\bigr\}
\end{equation}
consisting of the nearest-neighbour transpositions. The quantum symmetry group $\QI(\widehat{\sS_n})$ turns out to be the doubling of $\widehat{\sS_n}$ (see Section \ref{doubl})  with respect to the order-two automorphism $\theta_n\colon\sS_n\to\sS_n$ given by $\theta_n(\sn_i)=\sn_{n-i},\;i=1,\dotsc,n-1$, with the action of this quantum group on $\CC[\sS_n]$ being precisely the one described in Subsection \ref{doublss} (by formula \eqref{coact}). Since the index $n$ will now be extensively used, all objects pertaining to the quantum isometry group of $\sS_n$ will be denoted by symbols decorated with the subscript $n$. Thus $\QG_n$ will denote the quantum isometry group $\QI(\widehat{\sS_n})$. The corresponding \cst algebra will be $\C(\QG_n)$ with comultiplication $\Delta_n\colon\C(\QG_n)\to\CC(\QG_n)\tens\CC(\QG_n)$ (note that since $\C(\QG_n)$ is finite dimensional, it is in fact a Hopf algebra). We have $\C(\QG_n)=\CC[\sS_n]\oplus\CC[\sS_n]$ and the injections of $\CC[\sS_n]$ onto the first and second coordinate in $\C(\QG_n)$ will be denoted by $\xi_n$ and $\eta_n$ respectively. Recall from Subsection \ref{doublss} that the action of $\QG_n$ on $\CC[\sS_n]$ is given by the map $\alpha_n\colon\CC[\sS_n]\to\CC[\sS_n]\tens\C(\QG_n)$ defined as
\[
\alpha_n(a)=\bigl(\id\tens\xi_n+\theta_n\tens[\eta_n\comp\theta_n]\bigr)\comp\Delta_{\CC[\sS_n]}.
\]
As before we have the embedding $\Xi_n\colon\CC[\sS_n]\ni{a}\mapsto\xi_n(a)+\eta_n\bigl(\theta_n(a)\bigr)\in\C(\QG_n)$ which satisfies
\[
(\Xi_n\tens\id)\comp\alpha_n=\Delta_n\comp\Xi_n.
\]

Consider now the linear map $\ph_n\colon\CC[\sS_n]\to\CC[\sS_{n+2}]$ defined on generators as
\[
\ph_n(\sn_i)=\overset{\scriptscriptstyle{n+2}}{s}_{\!\!i+1}.
\]
Then $\ph_n$ is an injective Hopf algebra map which is equivariant: $\theta_{n+2}\comp\ph_n=\ph_n\comp\theta_n$. Denote by $\psi_n$ the canonical extension of $\ph_n$ to a map $\C(\QG_n)\to\C(\QG_{n+2})$ (cf.~Subsection \ref{doublss}). Then the following diagram
\begin{equation}
\label{dn0}
\xymatrix{
\CC[\sS_n]\ar[d]_{\ph_n}\ar[rr]^-{\alpha_n}&&\CC[\sS_n]\tens\C(\QG_n)\ar[d]^{\ph_n\tens\psi_n}\\
\CC[\sS_{n+2}]\ar[rr]_-{\alpha_{n+2}}&&\CC[\sS_{n+2}]\tens\C(\QG_{n+2})
}
\end{equation}
commutes. Moreover it is easily seen that $\ph_n$ maps length-one elements (nontrivial linear combinations of the generators \eqref{gener}) to length-one elements. It turns out that is is not possible to construct such a map acting between consecutive permutation groups, which is related to the fact that the sequence of corresponding quantum symmetry groups does not fit into the framework of Theorem \ref{mainlimit}.

\begin{proposition}\label{non-existence}
Let $n$ be even. There does not exists a linear map $\chi_n\colon\CC[\sS_n]\to\CC[\sS_{n+1}]$ mapping length one elements to length one elements and such that there is a Hopf algebra morphism $\zeta_n\colon\C(\QG_n)\to\C(\QG_{n+1})$ such that the diagram
\begin{equation}
\label{dn1}
\xymatrix{
\CC[\sS_n]\ar[d]_{\chi_n}\ar[rr]^-{\alpha_n}&&\CC[\sS_n]\tens\C(\QG_n)\ar[d]^{\chi_n\tens\zeta_n}\\
\CC[\sS_{n+1}]\ar[rr]_-{\alpha_{n+1}}&&\CC[\sS_{n+1}]\tens\C(\QG_{n+1})
}
\end{equation}
commutes.
\end{proposition}

\begin{proof}
For $i=1,\dotsc,n-1$ we denote by $\sigman_i$ and $\taun_i$ the elements $\xi_n(\sn_i)$ and $\eta_n(\sn_i)$ respectively. With this notation we have the following expression for the comultiplication of $\C(\QG_n)$:
\begin{equation}\label{Delst}
\begin{split}
\Delta_n(\sigman_i)&=\sigman_i\tens\sigman_i+\taun_i\tens\taun_{n-i},\\
\Delta_n(\taun_i)&=\sigman_i\tens\taun_i+\taun_i\tens\sigman_{n-i}
\end{split}\;\;\;\;\;\; (i=1,\ldots,n-1).
\end{equation}
We will use analogous notation for $n+1$ and $i=1,\dotsc,n$.

Assume that there exist maps $\chi_n$ and $\zeta_n$ for which \eqref{dn1} is commutative and $\chi_n$ maps length one elements to length one elements. The element $\sn_{\frac{n}{2}}$ is mapped by $\alpha_n$ to $\sn_{\frac{n}{2}}\tens\bigl(\sigman_{\frac{n}{2}}+\taun_{\frac{n}{2}}\bigr)$. Let us denote $\sigman_{\frac{n}{2}}+\taun_{\frac{n}{2}}$ by $\cn$. It is a group-like element of $\C(\QG_n)$.

Since $\chi_n$ maps length one elements to length one elements, we have
\[
\chi\bigl(\sn_{\frac{n}{2}}\bigr)=\sum_{k=1}^n\lambda_{k}\snn_k,
\]
so that
\[
\alpha_{n+1}\bigl(\chi\bigl(\sn_{\frac{n}{2}}\bigr)\bigr)
=\sum_{k=1}^n\lambda_{k}\bigl(\snn_k\tens\sigmann_k+\snn_{n+1-k}\tens\taunn_{n+1-k}\bigr).
\]
It follows that
\[
\chi_n\bigl(\sn_{\frac{n}{2}}\bigr)\tens\zeta_n\bigl(\cn\bigr)
=\sum_{k=1}^n\lambda_{k}\bigl(\snn_k\tens\sigmann_k+\snn_{n+1-k}\tens\taunn_{n+1-k}\bigr)
\]
which we rewrite as
\[
\begin{split}
\sum_{k=1}^n\snn_k\tens\lambda_{k}\zeta_n\bigl(\cn\bigr)
&=\sum_{k=1}^n\lambda_{k}\bigl(\snn_k\tens\sigmann_k+\snn_{n+1-k}\tens\taunn_{n+1-k}\bigr)\\
&=\sum_{k=1}^n\snn_k\tens\bigl(\lambda_{k}\sigmann_k+\lambda_{n+1-k}\tens\taunn_k\bigr)
\end{split}
\]
Clearly this implies that for each $k=1,\dotsc,n$ we have
\[
\lambda_{k}\zeta_n\bigl(\cn\bigr)
=\lambda_{k}\sigmann_k+\lambda_{n+1-k}\tens\taunn_k
\]
and the element on the right hand side is group-like. There must exist a $k$ for which $\lambda_k$ is not zero (otherwise the length-one element $\sn_{\frac{n}{2}}$ would be mapped to $0$ which is not length-one). Dividing out by this $\lambda_k$ and putting $\mu=\tfrac{\lambda_{n+1-k}}{\lambda_k}$ we obtain a group-like element
\[
\zeta_n\bigl(\cn\bigr)=\sigmann_k+\mu\taunn_k
\]
of $\C(\QG_{n+1})$. But by \eqref{Delst}
\[
\Delta\bigl(\sigmann_k+\mu\taunn_k\bigr)
=\sigmann_k\tens\sigmann_k+\taunn_k\tens\taunn_{n+1-k}+
\mu\bigl(\sigmann_k\tens\taunn_k+\taunn_k\tens\sigmann_{n+1-k}\bigr)
\]
while
\[
\bigl(\sigmann_k+\mu\taunn_k\bigr)\tens\bigl(\sigmann_k+\mu\taunn_k\bigr)=
\sigmann_k\tens\sigmann_k+\mu^2\taunn_k\tens\taunn_k+
\mu\bigl(\sigmann_k\tens\taunn_k+\taunn_k\tens\sigmann_k\bigr).
\]
It follows that $\mu^2$ must be equal to $1$, but more importantly, $n+1-k=k$. This is not possible for even $n$.
\end{proof}

It is worth noting that if one relaxes the requirement to preserve length-one elements, then a map satisfying other conditions stated in Proposition \ref{non-existence} exists. Indeed, for each $k\in\NN$ the group homomorphism $\gamma_{2k}\colon\sS_{2k}\to\sS_{2k+1}$ given by
\[
\gamma_{2k}(\skk_i)=
\begin{cases}
\skkj_i&1\leq{i}\leq{k-1},\\
\skkj_k\skkj_{k+1}\skkj_k&i=k,\\
\skkj_{i+1}&k+1\leq{i}\leq{2k-1}.
\end{cases}
\]
is an equivariant injection.

\subsection{Inductive limits} \label{indlim}

Let us consider the two inductive systems of \cst-algebras:
\begin{equation}\label{sys}
\begin{split}
&\xymatrix@1{\CC[\sS_1]\ar[r]^{\ph_1}&\CC[\sS_3]\ar[r]^{\ph_3}&\CC[\sS_5]\ar[r]^{\ph_5}&\dotsm}\,,\\
&\xymatrix@1{\CC[\sS_2]\ar[r]^{\ph_2}&\CC[\sS_4]\ar[r]^{\ph_4}&\CC[\sS_6]\ar[r]^{\ph_6}&\dotsm}\,
\end{split}
\end{equation}
where $(\ph_{2l+1})_{l\in\ZZ_+}$ and $(\ph_{2k})_{k\in\NN}$ are the morphisms introduced in Subsection \ref{qisoSn}. Both inductive limits are isomorphic to $\cst(\sS_\infty)$. On each component we have coactions of $\QG_n$ with appropriate $n$. Let us denote by $\Phi_n$ the canonical map $\CC[\sS_n]\to\cst(\sS_\infty)$. Since in both the systems \eqref{sys} the connecting maps are equivariant for the automorphisms $\theta_n$ of $\CC[\sS_n]$, we obtain automorphisms $\theta_\infty^\odd$ and $\theta_\infty^\ev$ of $\cst(\sS_\infty)$ such that
\[
\theta_\infty^\odd\comp\Phi_{2l+1}=\Phi_{2l+1}\comp\theta_{2l+1}\quad\text{and}\quad
\theta_\infty^\ev\comp\Phi_{2k}=\Phi_{2k}\comp\theta_{2k}
\]
for each $l\in\ZZ_+$, $k\in\NN$.

The  morphisms $\psi_n\colon\C(\QG_n)\to\C(\QG_{n+2})$, introduced in the paragraph before diagram \ref{dn0}, form the following inductive systems:
\begin{subequations}
\begin{align}
&\xymatrix@1{\C(\QG_1)\ar[r]^{\psi_1}&\C(\QG_3)\ar[r]^{\psi_3}
&\C(\QG_5)\ar[r]^{\psi_5}&\dotsm}\,,\label{Alim1}\\
&\xymatrix@1{\C(\QG_2)\ar[r]^{\psi_2}&\C(\QG_4)\ar[r]^{\psi_4}
&\C(\QG_6)\ar[r]^{\psi_6}&\dotsm}\,.\label{Alim2}
\end{align}
\end{subequations}

These are systems compact quantum groups with quantum group morphisms and so, by Theorem \ref{indqg} the respective limits are canonically endowed with compact quantum group structures. Clearly the \cst-algebra inductive limits $\varinjlim\cA_{2l+1}$ and $\varinjlim\cA_{2k}$ are both isomorphic to $\cst(\sS_\infty)\oplus\cst(\sS_\infty)$ which we will for the time being denote by $\sA_\infty$. Let us also denote the canonical maps $\C(\QG_n)\to\sA_\infty$ by $\Psi_n$. The \cst-algebra $\sA_\infty$ has two comultiplications $\Delta_\infty^\odd$ and $\Delta_\infty^\ev$ such that
\[
\Delta_\infty^\odd\comp\Psi_{2l+1}=(\Psi_{2l+1}\tens\Psi_{2l+1})\comp\Delta_{2l+1}\quad\text{and}\quad
\Delta_\infty^\ev\comp\Psi_{2k}=(\Psi_{2k}\tens\Psi_{2k})\comp\Delta_{2k}
\]
for each $l\in\ZZ_+$, $k\in\NN$. Let us denote the two quantum groups $(\sA_\infty,\Delta_\infty^\odd)$ and $(\sA_\infty,\Delta_\infty^\ev)$ by $\QG_\infty^\odd$ and $\QG_\infty^\ev$ respectively. By \eqref{doublH} and \cite[Proposition 3.3]{Wangfree} the Haar measures of $\QG_\infty^\odd$ and $\QG_\infty^\ev$ are actually both equal to the sum $\tfrac{1}{2}(\tau\comp\pr_1+\tau\comp\pr_2)$, where $\tau$ is the canonical (von Neumann) trace on $\cst(\sS_\infty)$ and $\pr_1$ and $\pr_2$ are projections of $\sA_\infty$ onto first and second direct summand. Let us denote this functional by $h_\infty$. It follows that $h_\infty$ is faithful.

The canonical Hopf $*$-algebras inside $\sA_\infty$ for $\QG_\infty^\odd$ and $\QG_\infty^\ev$ are the algebraic direct limits of \eqref{Alim1} and \eqref{Alim2} respectively (this follows e.g.~from uniqueness of this Hopf algebra, cf.~\cite[Theorem 5.1]{bmt}).

By Proposition \ref{indactQG} the diagram
\[
\xymatrix@C=3em{
\dotsm\ar[r]
&\CC[\sS_{n-2}]\ar[ddrrrrr]_(.55){\scriptscriptstyle{(\Phi_{n-2}\tens\Psi_{n-2})\circ\alpha_{n-2}}\quad}\ar[r]^{\ph_{n-2}}
&\CC[\sS_{n}]\ar[ddrrrr]^(.3){\scriptscriptstyle{(\Phi_{n}\tens\Psi_{n})\circ\alpha_{n}}}\ar[r]^{\ph_{n}}
&\CC[\sS_{n+2}]\ar[ddrrr]^-{\quad\scriptscriptstyle{ (\Phi_{n+2}\tens\Psi_{n+2})\circ\alpha_{n+2} }}\ar[r]
&\dotsm&\dotsm\ar[r]&\mathrm{C}^*(\sS_\infty)\\
&&&&&&\!\!\!\!\!\!\!\!\!\!\!\!\!\!\!\dotsm\,\,\,\,\,\,\,\,\,\,\,\,
\\
&&&&&&\mathrm{C}^*(\sS_\infty)\tens\sA_\infty
}
\]
gives rise to $\alpha_\infty^\odd\colon\cst(\sS_\infty)\to\cst(\sS_\infty)\tens\sA_\infty$ and
$\alpha_\infty^\ev\colon\cst(\sS_\infty)\to\cst(\sS_\infty)\tens\sA_\infty$ depending on the parity of the indices. It is easy to check that these are actions of $\QG_\infty^\odd$ and $\QG_\infty^\ev$ on $\cst(\sS_\infty)$. Moreover
\[
(\Phi_{2l+1}\tens\Psi_{2l+1})\comp\alpha_{2l+1}=\alpha_\infty^\odd\comp\Phi_{2l+1},\qquad
(\Phi_{2k}\tens\Psi_{2k})\comp\alpha_{2k}=\alpha_\infty^\ev\comp\Phi_{2k}
\]
for all $l\in\ZZ_+$, $k\in\NN$.

The two quantum dynamical systems $\bigl(\cst(\sS_\infty),\QG_\infty^\odd,\alpha_\infty^\odd\bigr)$ and $\bigl(\cst(\sS_\infty),\QG_\infty^\ev,\alpha_\infty^\ev\bigr)$ are embeddable. Indeed, for each $n$ we have the commutative diagram:
\[
\xymatrix@R=1ex@C=3em{
&&\CC[\sS_n]\tens\C(\QG_n)\ar[dddddd]^{\Xi_n\tens\id}\ar[rd]^(.45){\ph_n\tens\psi_n}\\
&&&\CC[\sS_{n+2}]\tens\C(\QG_{n+2})\ar[dddd]^{\Xi_{n+2}\tens\id}\\
\CC[\sS_n]\ar[r]^{\ph_n}\ar[dd]_{\Xi_n}\ar@<1ex>[rruu]^{\alpha_n}
&\CC[\sS_{n+2}]\ar[dd]^{\Xi_n}\ar[rru]_(.3){\alpha_{n+2}}\\\\
\C(\QG_n)\ar[r]^{\psi_n}\ar@<-1ex>[rrdd]_{\Delta_n}&\C(\QG_{n+2})\ar[rrd]^(.6){\Delta_{n+2}}\\
&&&\C(\QG_{n+2})\tens\C(\QG_{n+2})\\
&&\C(\QG_n)\tens\C(\QG_n)\ar[ru]_{\psi_n\tens\psi_n}
}
\]
It follows easily from properties of inductive limits that we have two injective morphisms
\[
\Xi_\infty^\odd\colon\cst(\sS_\infty)\longrightarrow\sA_\infty,\qquad
\Xi_\infty^\ev\colon\cst(\sS_\infty)\longrightarrow\sA_\infty
\]
which satisfy
\[
(\Xi_\infty^\odd\tens\id)\comp\alpha_\infty^\odd=\Delta_\infty^\odd\comp\Xi_\infty^\odd\quad\text{and}
\quad(\Xi_\infty^\ev\tens\id)\comp\alpha_\infty^\ev=\Delta_\infty^\ev\comp\Xi_\infty^\ev.
\]

One can see that the structures obtained in the two limits (over even and odd integers respectively), i.e.\ $\QG_{\infty}^{\ev}$ and $\QG_{\infty}^{\odd}$, are in fact doublings of the quantum group $\widehat{\sS_\infty}$ with respect to the order-two automorphisms $\theta_\infty^\ev$ and $\theta_\infty^\odd$. Quite clearly also the actions $\alpha_\infty^\ev$ and $\alpha_\infty^\odd$ and embeddings $\Xi_\infty^\ev$ and $\Xi_\infty^\odd$ arise via procedures described in Subsection \ref{doublss}. We can describe $\theta_\infty^\ev$ and $\theta_\infty^\odd$ in more detail: as automorphisms of the Hopf $*$-algebra $\bc[\sS_{\infty}]$ they arise from automorphisms of the group $S_{\infty}$, say $\kappa^{\ev}$ and $\kappa^\odd$. The latter in turn can be respectively identified, once we agree to identify $\sS_{\infty}$ with the group of finitely supported permutations of $\bz$, with the conjugations by the permutations $\pi^{\ev},\pi^\odd \in \textup{Perm}(\bz)$, given by the reflections of $\bz$ with respect to the point $\frac{1}{2}$ and $0$.

We will now prove that the quantum groups $\QG_\infty^\ev$ and $\QG_\infty^\odd$  are non-isomorphic. For that we first need to quote a well-known result due to Schreier and Ulam (\cite{Ulam}, see also Theorem 8.2.A of \cite{Dixon}).

\begin{theorem}\label{conjugate}(\cite{Ulam})
Any automorphism of the group $\sS_{\infty}$, interpreted as a group of finitely supported permutations of a given infinite countable set $X$, is given by a conjugation by a permutation (possibly with infinite support) of $X$.
\end{theorem}

In fact the groups $\textup{Aut}(\sS_{\infty})$ and $\textup{Perm}(X)$ are isomorphic, but this stronger statement will not be needed here.
We are now ready to state and prove the non-isomorphism result. Recall the definitions of $\kappa^\ev, \kappa^\odd, \pi^\ev$ and $\pi^\odd$ introduced in the paragraph before Theorem \ref{conjugate}.

\begin{theorem}\label{nonisomThm}
Compact quantum groups  $\QG_\infty^\ev$ and $\QG_\infty^\odd$ are not isomorphic.
\end{theorem}

\begin{proof}
To prove the theorem we need to show that the Hopf $*$-algebras arising as doublings of $\bc[\sS_{\infty}]$ induced by the automorphisms $\theta_\infty^\ev$ and $\theta_\infty^\odd$  are not isomorphic. We will argue by contradiction. As $\sS_{\infty}$ is an ICC group, Proposition \ref{isomDoubl} implies that if $\QG_\infty^\ev$ and $\QG_\infty^\odd$ are isomorphic, then there exists an automorphism of $\sS_{\infty}$ intertwining the automorphisms $\kappa^\ev$ and $\kappa^\odd$.
By Theorem \ref{conjugate} this means that (say we identify $\sS_{\infty}$ with the group of finitely supported permutations of $\bz$) there exists a permutation $\pi$ of $\bz$ such that for each $\sigma\in\sS_{\infty}$ there is
\[ \pi \kappa^\ev (\sigma) \pi^{-1} = \kappa^\odd (\pi \sigma \pi^{-1}).\]
Recalling that $\kappa^\ev, \kappa^{\odd}$ are given by conjugations with permutations $\pi^\ev$ and $\pi^{\odd}$, we deduce that the permutations  $\pi^\ev$ and $\pi^{\odd}$ are themselves conjugate (via $\pi$). This however is not possible, as $\pi^{\odd}$ fixes a point in $\bz$ and $\pi^{\ev}$ does not.
\end{proof}

\subsection{Interpretation of $\QG_\infty^\odd$ and $\QG_\infty^\ev$  as the quantum symmetry group of $\widehat{\sS_{\infty}}$ preserving certain partitions}

The group $\sS_{\infty}$ is generated by transpositions, and although the latter do not form a finite generating set, we will still denote the resulting word-length function on $\sS_{\infty}$ by $l$. As $l$ has now infinite level sets, we need some further idea to produce a partition of  $\sS_{\infty}$ into finite subsets. To this end we interpret $\sS_{\infty}$ as the group of all finite permutations of $\bz$ and introduce the following function $t:S_{\infty} \to \NN_{0}$: for each $\sigma  \in \sS_{\infty}$ put $t(\sigma)=\min\{n\in \NN_0: \sigma(k)=k \textup{ if } |k|>n\}$. Consider further the following collection of sets:
\[ F_{n,m}=\{\sigma \in \sS_{\infty}: l(\sigma) = n, t(\sigma)=m\}, \;\;\; n, m \in \NN_0.\]
It is easy to check that $\mathcal{F}=\{F_{n,m}:n,m\in \NN_0\}$ forms a partition of $\sS_{\infty}$ into finite sets.

\begin{theorem} \label{oddSymmetry}
The quantum group $\QG_\infty^\odd$ is isomorphic to the quantum symmetry group of $\widehat{\sS_{\infty}}$ preserving the partition $\mathcal{F}$ of $\sS_\infty$ introduced in the paragraph above.
\end{theorem}

\begin{proof}
Denote the quantum symmetry group of $\widehat{\sS_{\infty}}$ preserving the partition $\mathcal{F}$ by $\QG$, and let its canonical action (in the algebraic form) on the dual of $\sS_{\infty}$ be given by $\alpha:\bc[\sS_{\infty}]\to \bc[\sS_{\infty}]\odot \Pol(\QG)$ (the fact that $\bc[\sS_{\infty}]$ is the domain of the algebraic version of this action can be deduced from the results of \cite[Section 2]{orth}, cf.~also \cite[Proposition 2.2]{dz}).
Recall that on the other hand the action of $\QG_\infty^\odd$ on the dual of $\sS_{\infty}$, $\alpha_\infty^\odd:  \bc[\sS_{\infty}]\to \bc[\sS_{\infty}]\odot \Pol(\QG_\infty^\odd)$, is given (see formula \ref{coact} and the notations introduced in Subsection \ref{indlim}) by the prescription
\[ \alpha_\infty^\odd = (\id \ot \xi + \kappa^{\odd} \ot \eta \circ \kappa^{\odd})\circ \Com,\]
where $\Com$ denotes the coproduct of $\bc[\sS_{\infty}]$. Thus to check that the action $\alpha_\infty^\odd$ preserves the partition $\mathcal{F}$ it suffices to verify that $\kappa^{\odd}$ preserves individual sets in $\mathcal{F}$. This however is easy to see, as $\kappa^{\odd}$ maps transpositions to transpositions, and also fixes the function $t$ introduced before the formulation of the theorem. Thus   $\QG_\infty^\odd$ acts on $\widehat{\sS_{\infty}}$ in the $\mathcal{F}$ preserving fashion and there exists a (unique) morphism $\pi$ from $\QG_\infty^\odd$ to $\QG$ such that
\begin{equation} (\id \ot \pi) \circ \alpha = \alpha_\infty^\odd.\label{morph1}\end{equation}
To obtain the existence of the inverse morphism  to $\pi$ we will exploit the universal property of $\QG_\infty^\odd$ as a projective limit.

We identify, for each $n\in \NN$, the group $\sS_{2n+1}$ with the group of permutations of the set $\{-n,\ldots,n\}$ and denote the resulting embedding of $\bc[\sS_{2n+1}]$ into $\bc[\sS_{\infty}]$ by $\varphi_{2n+1,\infty}$. Note that in this identification the transposition $\sigmarnn_i$ corresponds to the transposition swapping points $-n+i-1$ and $-n+i$, which in turn implies that we have (in the notation of Subsection \ref{qisoSn}) $\varphi_{2n+1,\infty} =  \varphi_{2n+3,\infty} \circ \varphi_{2n+1}$. This further implies that on the image of $\bc[S_{2n+1}]$ with respect to $\varphi_{2n+1,\infty}$ (contained in the image of $\bc[S_{2n+3}]$ with respect to $\varphi_{2n+3,\infty}$) the following equality holds:
\begin{equation} \label{varphiinv} \varphi_{2n+1} \circ (\varphi_{2n+1,\infty})^{-1}  = (\varphi_{2n+3,\infty} )^{-1}.\end{equation}
Observe now that as the action $\alpha$ of $\QG$ on $\bc[\sS_{\infty}]$ preserves the partition $\mathcal{F}$, it also preserves  the subspace $\varphi_{2n+1,\infty}(\bc[\sS_{2n+1}])$ (which can be identified as the span of these elements $\gamma \in \sS_{\infty}$ for which $t(\gamma)\leq n$). This means that we can consider the following action of $\QG$ on $\bc[\sS_{2n+1}]$, which we will denote by $\beta_{2n+1}$:
\[ \beta_{2n+1} = ((\varphi_{2n+1,\infty})^{-1} \ot \id) \circ \alpha \circ \varphi_{2n+1,\infty}.\]
Formula \eqref{varphiinv} implies that for all $n \in \NN$
\begin{equation} \label{betaint} \beta_{2n+3} \circ \varphi_{2n+1} = (\varphi_{2n+1} \ot \id) \circ \beta_{2n+1}.\end{equation}
As $\alpha$ preserves the length function given by transpositions, the action $\beta_{2n+1}$ preserves the transposition
-
induced length function of $\sS_{2n+1}$. Thus there is a (unique) morphism $\sigma_{2n+1}$ from $\QG$ to $\QG_{2n+1}$ such that
\[ (\id \ot \sigma_{2n+1}) \circ \alpha_{2n+1} = \beta_{2n+1},\]
where $\alpha_{2n+1}$ is the canonical action of $\QG_{2n+1}$ on $\bc[\sS_{2n+1}]$. As each $\varphi_{2n+1}$ is injective, it follows also that $\sigma_{2n+1}$ from $\QG$ to $\QG_{2n+1}$ is the unique morphism such that
\[ (\varphi_{2n+1} \ot \sigma_{2n+1}) \circ \alpha_{2n+1} = (\varphi_{2n+1} \ot \id) \circ \beta_{2n+1}.\]
We now want to show that the morphisms $(\sigma_{2n+1})_{n=1}^{\infty}$ form a compatible family with respect to the embeddings $\psi_{2n+1}:\C(\QG_{2n+1}) \to \C(\QG_{2n+3})$ (see Subsection \ref{qisoSn}). To this end consider the following computation
\begin{align*} (\varphi_{2n+1} \ot (\sigma_{2n+3} \circ \psi_{2n+1})) \circ \alpha_{2n+1} &= (\id \ot \sigma_{2n+3})\circ \alpha_{2n+3} \circ \varphi_{2n+1} \\&=
\beta_{2n+3} \circ \varphi_{2n+1} = (\varphi_{2n+1} \ot \id) \circ \beta_{2n+1},
\end{align*}
where in the first equality we used the commutativity of the diagram \eqref{dn0}, in the second the defining property of $\sigma_{2n+3}$, and in the third the relation \eqref{betaint}.  Due to the uniqueness property of $\sigma_{2n+1}$ mentioned above, we deduce that $\sigma_{2n+3} \circ \psi_{2n+1} = \sigma_{2n+1}$ for each $n \in \NN$. This, as $\QG_{\infty}^{\odd}$ is the projective limit of $(\QG_{2n+1})_{n=1}^{\infty}$ with respect to the connecting morphisms $(\psi_{2n+1})_{n=1}^{\infty}$, implies that there exists a unique morphism $\sigma$ from $\QG$ to $\QG_{\infty}$ such that for each $n\in \NN$ there is $\sigma_{2n+1} = \sigma \circ \psi_{2n+1, \infty}$.
It remains to check that $\sigma$ satisfies
\begin{equation} (\id \ot \sigma) \circ \alpha_\infty^\odd = \alpha.\label{morph2}\end{equation}
To this end fix $n\in \NN$ and consider the following computation:
\begin{align*}
(\id \ot \sigma) \circ \alpha_\infty^\odd \circ \varphi_{2n+1,\infty} &= (\id \ot \sigma) \circ (\varphi_{2n+1,\infty} \ot \psi_{2n+1,\infty})\circ \alpha_n
\\&= (\varphi_{2n+1,\infty} \ot \sigma_{2n+1})\circ \alpha_{2n+1} = (\varphi_{2n+1,\infty} \ot \id)\circ \beta_{2n+1} = \alpha \circ \varphi_{2n+1,\infty} \end{align*}
where the first equality follows from the properties of the limit actions described in Proposition \ref{indactQG}, whereas the second, the third and the fourth respectively from the defining properties of $\sigma$, $\sigma_{2n+1}$ and $\beta_{2n+1}$. As $\bigcup_{n\in\NN} \varphi_{2n+1,\infty}(\bc[\sS_{2n+1}])$ is dense in $C^*(\sS_\infty)$, the formula \eqref{morph2} follows.

Finally, as both actions $\alpha$ and $\alpha_\infty^\odd$ are faithful, respectively by Theorem \ref{qsymV} and Proposition \ref{indactQG}, equalities \eqref{morph1} and \eqref{morph2} imply that the morphisms $\sigma$ and $\pi$ are inverse to each other, we conclude that the quantum groups $\QG$ and $\QG_\infty^\odd$ are isomorphic.
\end{proof}

A similar result is true for the quantum group $\QG_\infty^\ev$. This time we need to consider a slightly different partition of $\sS_{\infty}$, which we will now briefly describe. Interpret $\sS_{\infty}$ as the group of all finite permutations of  $\bz\setminus\{0\}$ and define the following function  $t':S_{\infty} \to \NN_{0}$: for each $\sigma  \in \sS_{\infty}$, we put $t'(\sigma)=\min\{n\in \NN_0: \sigma(k)=k \textup{ if } |k|>n\}$. Note that $t'$ is indeed a different function than $t$ introduced before! Further consider the following collection of sets:
\[ F'_{n,m}=\{\sigma \in \sS_{\infty}: l(\sigma) = n, t'(\sigma)=m\}, \;\;\; n, m \in \NN_0.\]
It is easy to check that $\mathcal{F'}=\{F'_{n,m}:n,m\in \NN_0\}$ forms a partition of $\sS_{\infty}$ into finite sets.

The following theorem can be proved in the same way as Theorem \ref{oddSymmetry}. We leave the details to the reader.

\begin{theorem} \label{evSymmetry}
The quantum group $\QG_\infty^\ev$ is isomorphic to the quantum symmetry group of $\widehat{\sS_{\infty}}$ preserving the partition $\mathcal{F'}$ of $\sS_\infty$ introduced above.
\end{theorem}


\begin{thebibliography}{999999}

\bibitem[Ban]{banmet} T. Banica, Quantum automorphism groups of small metric spaces, {\em Pacific J. Math.} {\bf 219} (2005), no.\,1, 27--51.

\bibitem[BaB]{BanBic} T. Banica and J. Bichon, Quantum automorphism groups of vertex-transitive graphs of order $\leq 11$, \emph{J. Algebraic Combin.} \textbf{26} (2007), no.\,1, 83--105.

\bibitem[BS$_1$]{bsk1} T. Banica and A. Skalski, Two-parameter families of quantum symmetry groups,  {\em J. Funct. Anal.} {\bf 260} (2011), no.\,11, 3252--3282.


\bibitem[BS$_2$]{bsk2} T. Banica and A. Skalski, Quantum isometry groups of duals of free powers of cyclic groups,  {\em Int.\,Mat.\,Res.\,Not.}  \textbf{9} (2012), no.\,6,  2094--2122.


\bibitem[BS$_3$]{orth} T.\,Banica and A. Skalski, Quantum symmetry groups of \cst-algebras equipped with orthogonal filtrations. To appear in \emph{Proceedings of the London Mathematical Society}, preprint available on arXiv as \texttt{arXiv:1109.6184 [math.OA]}.

\bibitem[BMT]{bmt}
E.~Bedos, G.J.~Murphy and L.~Tuset, Co-amenability for compact quantum groups, \emph{J.~Geom.~Phys.} \textbf{40} (2001),  no.~2, 130--153.



\bibitem[BhG]{Deb2} J.~Bhowmick and D.~Goswami, Quantum group of orientation preserving Riemannian Isometries, \emph{J.~Funct. Anal.} \textbf{257}  (2009), no.~8, 2530--2572.

\bibitem[BGS]{bgs} J.~Bhowmick, D.~Goswami and A.~Skalski, Quantum isometry groups of 0-dimensional manifolds, \emph{Trans. Amer.~Math.~Soc.} \textbf{363} (2011),  no.\,2, 901--921.


\bibitem[BhS]{bhs} J.\,Bhowmick and A.\,Skalski, Quantum isometry groups of noncommutative manifolds associated to group ${\rm C}^*$-algebras, {\em J.\,Geom.\,Phys.} {\bf 60} (2010), no.\,10, 1474--1489.


\bibitem[Bla]{bl} B.~Blackadar: \emph{Operator algebras. Theory of $\mathrm{C}^*$-algebras and von Neumann algebras.} Encyclopedia of Mathematical Sciences,
    Vol.~\textbf{122}, Springer-Verlag 2006.


\bibitem[DCh]{Manon} M.T.\,De Chanvalon, Quantum symmetry groups of Hilbert modules equipped with orthogonal filtrations, preprint available on arXiv as \texttt{arXiv:1304.3718 [math.OA]}.

\bibitem[DKSS]{DKSS} M.\,Daws, P.\,Kasprzak, A.\,Skalski and P.\,So\l tan, Closed quantum subgroups of locally compact quantum groups, \emph{Adv.\,Math.} \textbf{231} (2012), 3473--3501.


\bibitem[DiK]{DiK}
M.\,Dijkhuizen and T.\,Koornwinder,
CQG algebras --- a direct algebraic approach to compact quantum groups,
\emph{Lett.\,Math.\,Phys.}
\textbf{32} (1994), no.\,4, 315--330.

\bibitem[DMo]{Dixon} J.D.\,Dixon and B.\,Mortimer: \emph{Permutation groups.} Graduate Texts in Mathematics, Vol.~\textbf{163}. Springer-Verlag, New York, 1996.




\bibitem[Dri]{Drinfeld} V.G.~Drinfeld, Quantum groups, in \emph{Proc.~Int.~Congr.~Mathematicians} (Berkeley, CA, USA, 1986), pp.~793--820.

\bibitem[Gos]{gos} D.\,Goswami, Quantum group of isometries in classical and noncommutative geometry, {\em Comm. Math. Phys.} {\bf 285} (2009),  no.\,1, 141--160.


\bibitem[L-DS]{Symmetric} J.~Liszka-Dalecki and P.M.~So{\l}tan, Quantum isometry groups of symmetric groups, \emph{Int.~J.~Math.} \textbf{23}  (2012), no.\,7, 1250074-1--1250074-25.

\bibitem[Pod]{podles}
P.~Podle\'{s}, Symmetries of quantum spaces. Subgroups and quotient spaces of quantum $\mathrm{SU}(2)$ and $\mathrm{SO}(3)$ groups, \emph{Commun.~Math.~Phys.} \textbf{170} (1995),  no.\,1, 1--20.


\bibitem[RLL]{Niels} M.~R\o rdam, F.Larsen and N.Laustsen: \emph{An Introduction to $K$-theory for $\mathrm{C}^*$-algebras.} London Mathematical Society Student Texts, Vol.~\textbf{49}, Cambridge University Press, Cambridge, 2000.

\bibitem[ScU]{Ulam} J.\,Schreier and S.\,Ulam,  \"Uber die Automorphismen der Permutationsgruppe der nat\"urlichen Zahlenfolge, \emph{Fundam.\,Math.}  \textbf{28} (1936), 258–-260.

\bibitem[So{\l}]{dz}
P.M.~So{\l}tan, On actions of compact quantum groups. To appear in \emph{Illinois Journal of Mathematics,} preprint available on arXiv as {\tt arXiv:1003.5526v2 [math.OA].}




\bibitem[TQi]{Chin}  J.\,Tao and D.\,Qiu, Quantum isometry groups for dihedral group $D_{2n (n+1)}$, \emph{J.\,Geom.\,Phys.} \textbf{62} (2012),  no.\,9, 1977--1983.

\bibitem[VDa]{Algebraic} A.~Van Daele, An algebraic framework for group duality, \emph{Adv.~Math.} \textbf{140} (1998), 323--366.


\bibitem[Wa$_1$]{WangSym} S.\,Wang, Quantum symmetry groups of finite spaces, {\em Comm. Math. Phys.} {\bf 195} (1998),  no.\,1, 195--211.

\bibitem[Wa$_2$]{Wangfree} S.~Wang, Free products of compact quantum groups, \emph{Commun.~Math.~Phys.} \textbf{167} (1995), no.\,3, 671--692.



\bibitem[Wor]{wor}
S.L.\,Woronowicz, Compact quantum groups, in  {\em
  Sym{\'e}tries Quantiques, Les Houches, Session LXIV, 1995}, pp.\ 845--884.



\end{thebibliography}
\end{document}